\documentclass[11pt]{article}
\usepackage{sibarticle}
\usepackage{appendix}
\usepackage{natbib}
\usepackage{setspace}
\usepackage{xcolor, color}
\usepackage{mdframed}

\setuptodonotes{backgroundcolor=red!10, linecolor=red,
  bordercolor=red, size=\tiny, shadow, tickmarkheight=0.1cm}

\newcommand{\bm}[1]{\mbox{\boldmath{$#1$}}}

\newcommand{\dsum}{\displaystyle\sum}

\newcommand{\CB}{\mathcal{B}}

\newcommand{\CJ}{\mathcal{J}}
\newcommand{\CK}{\mathcal{K}}

\newcommand{\CS}{\mathcal{S}}

\newcommand{\RR}{\mathbb{R}}

\newcommand{\vc}{\mathbf{c}}

\newcommand{\vr}{\mathbf{r}}
\newcommand{\vs}{\mathbf{s}}

\newcommand{\vu}{\mathbf{u}}
\newcommand{\vv}{\mathbf{v}}
\newcommand{\vw}{\mathbf{w}}
\newcommand{\vx}{\mathbf{x}}

\newcommand{\vz}{\mathbf{z}}
\newcommand{\zv}{\mathbf{0}}
\newcommand{\ov}{\mathbf{1}}

\newcommand{\veta}{\bm{\eta}}
\newcommand{\vxi}{\bm{\xi}}
\newcommand{\valpha}{\bm{\alpha}}
\newcommand{\vbeta}{\bm{\beta}}
\newcommand{\vgamma}{\bm{\gamma}}
\newcommand{\vtheta}{\bm{\theta}}
\newcommand{\vlambda}{\bm{\lambda}}
\newcommand{\vnu}{\bm{\nu}}

\newcommand{\mA}{\mathbf{A}}
\newcommand{\mB}{\mathbf{B}}

\newcommand{\mD}{\mathbf{D}}
\newcommand{\mE}{\mathbf{E}}
\newcommand{\mF}{\mathbf{F}}
\newcommand{\mG}{\mathbf{G}}
\newcommand{\mH}{\mathbf{H}}
\newcommand{\mI}{\mathbf{I}}
\newcommand{\mL}{\mathbf{L}}
\newcommand{\mM}{\mathbf{M}}

\newcommand{\mR}{\mathbf{R}}
\newcommand{\mS}{\mathbf{S}}
\newcommand{\mU}{\mathbf{U}}

\newcommand{\tr}{^{\intercal}}

\newcommand{\inv}{^{-1}}

\newcommand{\hspp}{\hspace{8mm}}

\newcommand{\maximize}{\mbox{maximize\hspace{4mm} }}
\newcommand{\minimize}{\mbox{minimize\hspace{4mm} }}
\newcommand{\subto}{\mbox{subject to\hspace{4mm}}}

\newcounter{commentcounter}
\long\def\symbolfootnote[#1]#2{\begingroup
  \def\thefootnote{\fnsymbol{footnote}}\footnote[#1]{#2}\endgroup}

\newcommand{\lemnum}[2]{\vspace{3mm}
  \noindent {\sc Lemma #1}{\it #2} \vspace{3mm}}
\newcommand{\thmnum}[2]{\vspace{3mm}
  \noindent {\sc Theorem #1}{\it #2} \vspace{3mm}}
\usepackage{algorithm, caption}
\usepackage[noend]{algpseudocode}

\newcommand{\mN}{\mathbf{N}}

\title{Masking Primal and Dual Models for \\ Data Privacy in Network Revenue Management}
\ShortTitle{Data Privacy in NRM}
\NumberOfAuthors{4}
\FirstAuthor{Utku Karaca}
\FirstAuthorAddress{Erasmus University
	Rotterdam, 3000 DR, Rotterdam P.O. Box 1738, The Netherlands}

\SecondAuthor{Ş. İlker Birbil}
\SecondAuthorAddress{University of
  Amsterdam, 11018 TV, Amsterdam P.O. Box 15953, The Netherlands}

\ThirdAuthor{Nurşen Aydın }
\ThirdAuthorAddress{University of Warwick,
  Coventry CV4 7AL, United Kingdom}

\FourthAuthor{Gizem Mullaoğlu \correspondingauthor}
\FourthAuthorAddress{Eindhoven
  University of Technology, 5600 MB Eindhoven, The Netherlands}


\allowdisplaybreaks

\begin{document}

\maketitle

\begin{abstract}
We study a collaborative revenue management problem where multiple decentralized parties agree to share some of their capacities. This collaboration is performed by constructing a large mathematical programming model available to all parties. The parties then use the solution of this model in their own capacity control systems. In this setting, however, the major concern for the parties is the privacy of their input data along with their individual optimal solutions. We first reformulate a general linear programming model that can be used for a wide-range of network revenue management problems. Then, we address the data-privacy concern of the reformulated model and propose an approach based on solving an equivalent data-private model constructed with input masking via random transformations. Our main result shows that after solving the data-private model, each party can safely access only its own optimal capacity control decisions. We also discuss the security of the transformed problem in the considered multi-party setting. We conduct simulation experiments to support our results and evaluate the computational efficiency of the proposed data-private model. Our work provides an analytical approach and insights on how to manage shared resources in a network problem while ensuring data privacy. Constructing and solving the collaborative network problem requires information exchange between parties which may not be possible in practice. Including data-privacy in decentralized collaborative network revenue management problems with capacity sharing is new to the literature and relevant to practice.
\end{abstract}

\section{Introduction.}
\label{sec:intro}

Forming alliances is an important business strategy for a firm to streamline its costs to remain competitive. Alliances can be considered as the collaboration among several parties to conduct various activities such as allocating resources, sharing information, and providing complementary services. These partnerships can also be observed among competitors, like several firms joining their professional assets to manage a supply chain network \citep{Granot05, He15}. In power networks, independent power suppliers operate on a distributed system where they work together to balance the power demand and supply \citep{Ghaderyan21}. Recently, logistics companies have started to collaborate by sharing empty vehicle capacities to overcome the problem of excess capacity in freight transportation \citep{speranza18}. One of the important advantages of collaboration is the increase in economies of scale. In airline revenue management, the carriers sign an alliance contract, called codeshare agreements, to share their flight capacities, and provide joint services. The companies can offer more products through joint services, leading to greater revenue opportunities in the long term \citep{topaloglu12}. A greater customer value can be achieved with increased flexibility in the provided services. In addition to economic benefits, the companies can also improve their reputation by improving the sustainability aspects of their operations \citep{Gansterer16}.

Coordination of partnerships or collaborations involves a series of challenges. Generally, in an alliance, parties pool some of their resources and share the necessary information for the collaborative decision making process. Besides coordinating shared resources, they also manage their individual local resources. Therefore, these parties, though often working towards a common goal, can be competitors and may be unwilling or unable to disclose complete information about their operations to protect their own interests. In addition, legal frameworks like the antitrust law force companies to take extra measures to protect their data \citep{Gerlach13, Wright14, Albrecht15}. Depending on the industry, this sensitive information may involve demand forecasts, selling prices, operational costs and available resources. For instance, in air-cargo transportation, airlines and freight forwarders collaborate to sell the flight capacity. In this partnership, the freight forwarders keep their demand information, operating costs, and reservation prices private to protect their interests \citep{amaruchkul11}. \cite{Gerlach13} report that sharing information obtained through dual variables in an airline alliance requires antitrust immunities, and in practice, airlines do not prefer to exchange such information to protect their own interests. Due to the restrictions in information exchange, decomposition and decentralization approaches have been studied to minimize information sharing between parties in collaborative networks \citep{poundarikapuram04, Albrecht15, Singh16, Ding19}. Research to date on decentralized collaborative decision-making problems has primarily adopted iterative negotiation-based approaches to decompose the centralized model, which requires complete information sharing. Although information exchange is reduced with the decentralized coordination and negotiation-based approaches, the parties still have to share some information on their individual operations, which may reveal confidential information about their activities. Several studies have pointed out that revealing primal or dual optimal solution can provide strategic advantage to other parties in the cooperation \citep{Albrecht15, Singh16, Lai2019, Sorolla20}. There is no mechanism available in the literature to coordinate independent cooperating parties while ensuring that the information shared by the parties remains private.

This paper considers a general setting for capacity collaboration in network revenue management problems, where applications can be found in airline alliances, air-cargo transportation \citep{Wright14, Houghtalen11}, collaborative logistics \citep{Adenso14, Gansterer16, Jin19}, decentralized supply chains \citep{Albrecht15, Singh16}. Considering the benefits of collaboration, we assume that multiple parties agree to collaborate by sharing some of their capacities. Each party also controls its own private capacities in addition to the shared resources. This collaboration is performed by constructing a large mathematical programming model available to all parties. The fundamental aim of the parties is to identify the optimal allocated capacities for the shared resources and to evaluate the opportunity costs of the capacities available to them, \textit{i.e.}, dual variables. These opportunity costs are used in various capacity control policies. For instance, dual variables are used in well-known \textit{bid-price control} policy to manage customer requests for quantity-based network revenue management problems \citep{Phillips05}. In our setting, the parties jointly compute the optimal capacity allocations and bid-prices of their collaborative model without disclosing any private information, such as selling prices and local capacity restrictions. However, without the necessary and mostly private information about the collaborative network problem, the correct values of bid-prices and the capacity allocations for the shared resources cannot be computed. This lack of proper information about the network problem raises an important question: How can one compute the correct bid-prices and the capacity allocations of the shared resources that maximize the overall revenue under privacy concerns? This question constitutes the main motivation behind our current study. Thus, by answering this question, we can provide a mechanism for the parties to collaborate without revealing their private data.

\paragraph{Contributions.} We present a new transformation-based approach that considers data privacy in collaborative network revenue management problem, where multiple parties agree to share some of the network capacities. The proposed approach allows partners to use their individual private data while solving the collaborative capacity control problem to identify the optimal capacity sharing setting for the alliance. In our setting, each party keeps its data private through random data masking. Unlike the previous decomposition methods proposed for decentralized collaborative resource sharing problems, this method does not require any unmasked information exchange among parties while solving the collaborative model. To the best of our knowledge, our approach is the first attempt in the literature to deal with data privacy, considering both primal and dual variables. Our analysis makes use of several previous privacy studies based on random transformations of the problem data. However, our focus on bid-prices allows us to extend these studies
with new results about the privacy of dual solutions. We show that the original primal and dual optimal solutions can be derived from the proposed transformed data-private model. Furthermore, in a separate section, we discuss the security of our mathematical model, where we apply a special set of random matrices ($M$-matrices) for transforming the simple inequalities. This set of matrices is much larger than the set of permutation matrices used in other studies which enhances the security of the proposed method. We also contribute to that literature with a new result showing that even if a private dataset of a firm is guessed, a brute-force approach to obtain random matrices in order to reveal primal and dual optimal solutions is computationally infeasible. We support our results with a simulation study on a set of revenue management problems, where the network structure is taken from an actual firm and adapted to an alliance network. Finally, we remark that the steps that we follow in this study can be extended to other resource sharing applications, where linear programming is one of the fundamental tools, and data privacy is a major concern. %such as distributed network utility maximization problems in wireless communication and power networks \citep{Karakoc20} \sitodo{Bu aşamada neden sadece tek bir makaleyi öne çıkarıyoruz anlamadım?}

\section{Review of Related Literature.}
\label{sec:lit}
Collaboration via forming alliances is common in many
industries. Therefore, the efficient capacity allocation among the
involved parties has long constituted an intriguing research
topic with applications in airlines \citep{wright2010dynamic, topaloglu12, Chun17}, logistics and maritime shipping \citep{Agarwal10, Zheng15, gansterer2018collaborative, Ding19}, and retail industries \citep{Guo18}. Previous literature on collaborative decision-making has primarily concentrated on the development of centralized models, which require information sharing among partners or with a central planner \citep{Boyd98, topaloglu12, Zheng15}. Recent studies have pointed out the impracticality of central planning due to the restrictions in information sharing. \cite{Belobaba13} and \cite{Wright14} have discussed the limitations in information sharing among alliance partners in airlines. Similarly, \cite{Albrecht15} have studied the collaboration in supply chains and pointed out that the current advanced supply chain planning systems assume complete information sharing between firms, and there is no mechanism to coordinate a system where the partners only share limited information. They have proposed a negotiation based scheme to coordinate collaborative parties in a decentralized supply chain environment. Due to limitations in information sharing, decentralized models have been studied for collaborative decision-making problems \citep{Sorolla20}. 

\cite{poundarikapuram04} have proposed a decentralized
decision-making framework based on the L-shaped method for a collaborative planning problem
in a supply chain. This framework separates the centralized problem into a master
problem that includes common variables for all parties and sub-problems with
private local objectives and variables. The authors have presented an iterative procedure, where the parties can solve their local problems privately and disclose limited information to solve the master problem at each iteration. \cite{topaloglu12} has studied a
centralized alliance problem in airline revenue management. By relaxing the shared constraints with dual variables, he has proposed a decomposition approach to find booking limits for each alliance partner as well as bid-prices. \cite{amaruchkul11} have addressed data privacy in air-cargo transportation. The carrier allocates bulk cargo capacity to the forwarder that sells this capacity to individual customers. The authors have studied the capacity contract between these two partners when the forwarder has private information on demand distribution, operating costs, and reservation profits. \cite{Chun17} have addressed the capacity
exchange problem in maritime transportation and proposed a two-stage framework to obtain the optimal resource allocation policy between alliance partners. In the first stage, optimal capacity exchange amount is determined so that the total alliance profit is maximized. Given the allocated capacities, each party decides on the reservation price to maximize its own local objective in the second stage. Recently, \cite{Lai2019} have studied the capacity allocation problem for a freight alliance by considering the data privacy in revenues and profit margins. They assume that the alliance partners jointly book freight capacity from the market according to the forecasted shipping
demand, and then share this capacity during the planning period. The authors have first studied the centralized capacity allocation problem assuming revenue information is public. Due to the difficulty in solving the centralized problem, they have proposed an iterative auction mechanism based on the primal-dual method to find the capacity allocation policy. In the designed auction mechanism, the alliance partners do not need to share their private data except the forecasted shipping demand.  

Although decomposition or decentralization approaches reduce the information exchange
among parties, they do not completely ensure data privacy. We review two main approaches in privacy preserving methods for
optimization and data analysis: cryptographic and non-cryptographic
approaches \citep{weeraddana2013per}. Cryptography-based techniques such as secure multiparty computation allow several parties to do computation jointly over their inputs while ensuring that they remain private to each party. There is
an extensive literature on cryptographic methods, yet computational
issues are always of concern. In fact, forming a completely secure
protocol for an optimization problem requires very high computational power
\citep{hong2018privacy}. \cite{li2006secure} have presented a secure
simplex algorithm for a setting, where the objective function and the
constraints are arbitrarily partitioned. \cite{toft2009solving} has followed the
same scheme with \citeauthor{li2006secure} and presented a protocol
for solving linear programs using a secure simplex algorithm. However,
\cite{dreier2011practical} have implemented the algorithm proposed by
\cite{toft2009solving} and pointed out that it is computationally
inefficient even for small-scale problems. 

Transformation is a non-cryptographic technique that involves
converting a given linear optimization problem into a new problem via
algebraic transformations, such that the solution of the new problem
is the same as that of the original problem
\citep{mangasarian2012privacy,wang2011secure}. This enables parties to
disguise private data effectively while preserving the quality of the solution. \cite{du2001study} and \cite{vaidya2009privacy} have used transformation method in linear programming models. However,
\cite{bednarz2012methods} has shown that this method is open to
information acquisition attacks; that is, the private coefficients in the model can be learned by others.
\cite{mangasarian2011privacy,mangasarian2012privacy} has proposed
transformation techniques for vertically and horizontally partitioned
linear programming models. \cite{li2013privacy} have
extended this approach by incorporating inequality constraints to
horizontally partitioned linear programming models. \cite{hong2014inference} have showed that the transformation method proposed by 
\cite{li2013privacy} is also open to attacks. 
There are few other transformation
approaches for privately solving collaborative linear programming
problems \citep{hong2018privacy,weeraddana2013per}. To the best of our
knowledge, all of these transformation approaches focus on privacy in
input data (private data) and primal optimal solution. How the dual solution is affected by the transformation applied to the primal model is unexplored.

\section{Data-Private Capacity Control.}
\label{sec:model}

Consider a collaborative network revenue management problem where multiple parties
cooperate to share several capacities of the network for their own
resource allocation systems. In addition to the shared capacities, each party controls its own private capacities. The aim of the parties is to decide on the optimal allocated capacities for the shared resources and to identify the bid-prices of the capacities available to them. In this setting, parties jointly build and solve the capacity sharing problem without disclosing any private information about their operations. Parties set
a partnership agreement at the beginning of the planning horizon. Therefore, a partner does not have any
access to the private data of any other partner in the cooperation. Depending on the industry, this private data may be revenues, available resources, operation routes, and demand information. 

Like other studies in the literature, we also assume that the parties
in an alliance cooperate truthfully \citep{Krajewska06, topaloglu12, hyndman2013aligning}. This assumption is known as semi-honest behaviour in computer science literature \citep{vaidya2009secure}. It implies that the involved parties do not
alter their data to get a better position in the collaboration. This
assumption is naturally satisfied for several important reasons: If
one of the parties strategically manipulate the collaboration and
secure most of the shared capacities, then there can be legal
consequences or simply loss-of-goodwill for their business. Firms must consider
the long-term prospects of future collaborations that can be jeopardized with an opportunistic
behaviour. Moreover,
as the optimal primal and dual solutions do not result from the correct network
information, they are of no use for controlling the shared
capacities. Such a harmful move from one of the parties can also cause
other parties to opt out of the partnership in the next round. Even
worse, the other parties may start altering their data, which also
eventually leads to total collapse of the collaboration. As a final
example, consider a holding company that owns all the parties in the
collaboration. Due to legal regulations, it may be impossible to share
data among the involved parties without transformation. In such a
setting, there is no incentive for any participant to manipulate the
collaboration at the expense of the others.

\subsection{Capacity Sharing Model.}

In this section, we first describe the general type of collaborative network revenue management problems that can be solved via our approach. We start with an illustration of the proposed capacity sharing setting. Figure \ref{fig:exnetwork} shows a simple network structure
for two parties. Each link between a pair of nodes corresponds to a shared or private capacity on the
network. The first party operates the capacities (1-3) and (3-4), whereas the second party operates the capacities (2-3) and (2-4). The second party also uses the capacity (3-4) shared by the first party. A sequence of links constitutes a path (\textit{e.g.} route or
itinerary). The set of paths listed in Figure \ref{fig:exnetwork}
shows a number of origin-destination combinations. Moreover, on each
path there could be multiple products with different revenues. An
example could be the set of different transshipment prices on a route
that involves several stops with different truck capacities on a
transportation network. \cite{birbil14} have also explored this
network structure and proposed a framework based on path
decomposition. They treat each path as a single resource problem for a
fixed capacity and solve an optimization problem over all possible
allocations of the capacities. We next reconsider the
path-based model of \cite{birbil14} and write it in a form that we can
analyze to propose a data-private capacity control approach.

\begin{figure}[ht]
 \centerline{
     \includegraphics[scale=0.35]{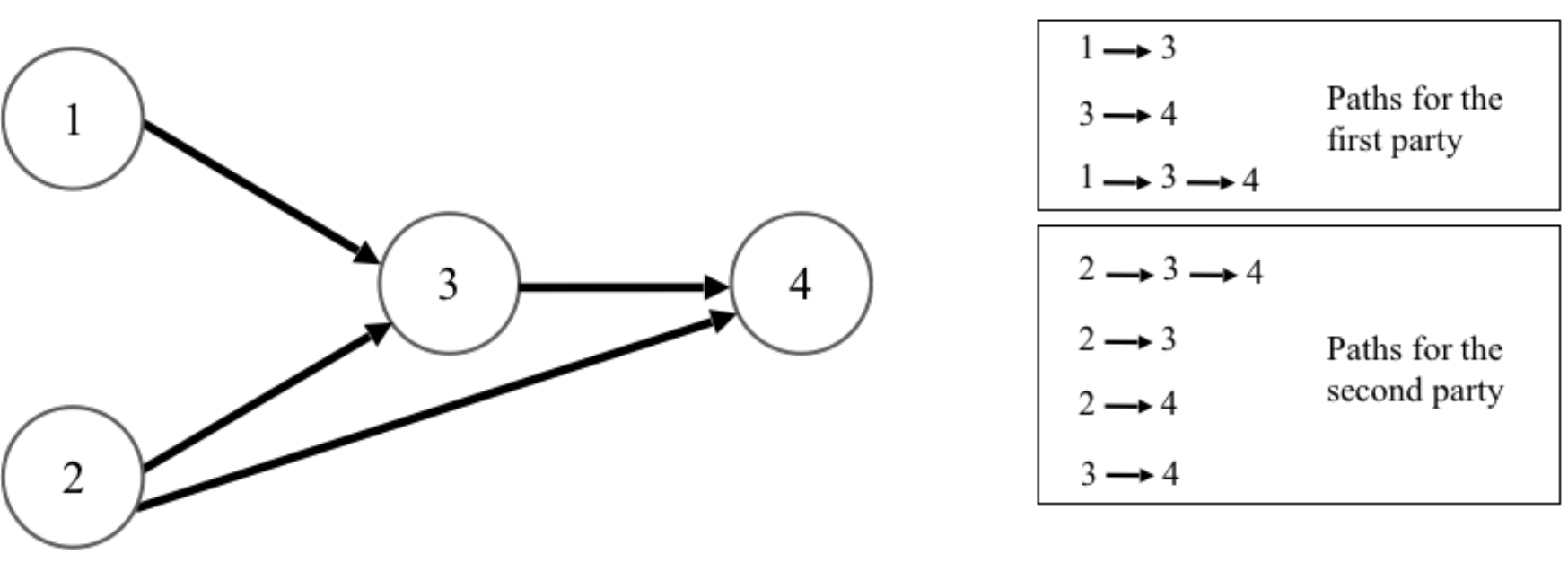}
   }
   \caption{An illustrative network structure for two parties. The
     capacity on path $1 \rightarrow 3$ is private for the first
     party, whereas the capacities on paths $2 \rightarrow 3$ and
     $2 \rightarrow 4$ are private for the second party. Two parties
     share the capacity on path $3 \rightarrow 4$.}
 \label{fig:exnetwork}
\end{figure}

Before we present our main capacity sharing model, let us give our
notation. We denote the set of parties by $\CK$ and the set of paths
controlled by party $k \in \CK$ is denoted by $\CS_k$. Let $\CJ$ be
the set of $m$ capacities shared by at least two parties. In addition, each party
$k \in \CK$ has its own $m_k$ private capacities given by the set $\CJ_k$. If the collection of paths is $\CS$ and $x_s$ is the allocated capacity to path $s \in \CS$, then the generic model becomes
\begin{equation}
\label{eq:model1}
\begin{aligned}
  \maximize & \sum_{s \in \CS}  \phi_{s}(x_{s}), \\
  \subto & \sum_{s\in\CS} a_{js}x_{s} \leq c_{j},& j \in \CJ, \\
            & \sum_{s \in \CS_k} a_{js}x_{s} \leq c_{j},& \qquad j \in \CJ_k, k\in \CK,\\
            & x_{s} \geq 0, & s \in \CS,
\end{aligned}
\end{equation}
where $a_{js} = 1$, if path $s$ uses one unit from capacity $j$; 0,
otherwise.  The shared capacities are denoted by $c_j$, $j\in\CJ$, and
$c_j$, $j\in\CJ_k$ stand for the private capacities of party
$k \in \CK$. The first set of constraints ensures that the capacity
allocation decision for paths do not violate the shared
capacities. The second set of constraints guarantees that the capacity
allocation decisions to paths for party $k \in \CK$ do not exceed the
private capacity limits for that party. The function
$ \phi_{s}(x_{s})$ for a given $x_s$ is itself evaluated by solving an
optimization problem that yields the allocation decision of $x_s$
capacities to different classes with the objective of maximizing
revenue. For instance, $\phi_s(x_s)$ may correspond to a stochastic
dynamic programming or a deterministic programming model constructed
for capacity allocation problem for each path $s$. We assume that
$x_s \mapsto \phi_s(x_s)$, $s \in \CS$ are discrete concave
functions. \cite{birbil14} have discussed that many well-known single
dimension capacity control models proposed in the revenue management
literature satisfy this assumption. Any of these models can be used to
construct the objective function. We refer to \cite{birbil14} for an
elaborate discussion on different dynamic and static network revenue
management problems that can be considered within this generic
structure.

As the objective function is concave and separable, we can replace it
by a piece-wise linear concave function and reformulate the problem as
a linear program. \cite{dantzig} has proposed an approach which
represents the concave objective function as an indefinite integral
and approximates it by a sum over fixed intervals. Following this approach, 
we can reformulate the model (\ref{eq:model1}) as follows:
\begin{align}
  Z = ~ & \maximize && \hspace*{-20mm}\sum_{k \in \CK} \vr_k\tr \vx_k,     \label{eqn:org_model_obj} \\
       	& \subto    && \hspace*{-20mm}\sum_{k \in \CK} \mA_k\vx_k \leq \vc, &&             && \hspp (\valpha) \label{eqn:org_model_c} \\
       	&           && \hspace*{-20mm}\mB_k\vx_k \leq  \vc_k,               && k \in \CK,  && \hspp (\valpha_k) \label{eqn:org_model_ck}\\[2mm]
       	&           && \hspace*{-20mm}\zv \leq \vx_k \leq \ov, && k \in \CK, && \label{eqn:org_model_nonneg}
\end{align}
where $\ov$ and $\zv$ stand for the vector of ones and the vector of
zeros, respectively. The details of this model construction is given in Appendix \ref{sec:app1}. In this model, the columns designated by subscript
$k \in \CK$ show all products owned by party $k$ and the vector
$\vr_k \in \RR^{n_k}$ denotes the corresponding expected revenues for
the same party. The $m \times n_k$ incidence matrix $\mA_k$ shows
whether a product of party $k$ uses the shared capacities. Likewise,
the $m_k \times n_k$ matrix $\mB_k$ consists of columns incident to
the private capacities. The vectors $\valpha \in \RR^m$ and
$\valpha_k \in \RR^{m_k}$, $k \in \CK$ given in parentheses are the
dual variables (bid-prices) associated with the common and the
individual capacity constraints, respectively. After solving this
problem, each party obtains its own optimal allocations and
bid-prices. They also receive the optimal common bid-price vector
corresponding to the shared capacities. These bid-prices can be used
by the parties to implement their decision-making policies. Therefore, the decision vector
$\vx_k$ and individual dual variables $\valpha_k$ are private to party $k \in \CK$.
 In our subsequent discussion,
the optimal values of capacity allocations and bid-prices for each
party and the optimal bid-price vector for shared capacities are
denoted by $\vx_k^*$, $\valpha_k^*$ and $\valpha^*$. Before discussing 
the data-private mathematical model, let us define formally what constitutes as the private dataset
for each party.

\begin{definition}
	\label{def:dataset}
	In multi-party capacity sharing problem	\eqref{eqn:org_model_obj}-\eqref{eqn:org_model_nonneg}, the private dataset
	for party $k \in \CK$ consists of the matrices $\mA_k$,
	$\mB_k$, and the vectors $\vc_k$, $\vr_k$.
\end{definition}

One may question why multiple parties would prefer solving the network
problem collectively? Instead, the shared capacity, $\vc$ may be
partitioned among the parties and each party can solve its problem independently. That is, each party $k \in \CK$ receives
its share of the capacity denoted by $\vs_k$ such that
$\vc = \sum_{k \in \CK} \vs_k$. Then, the same party $k$ can solve the
following problem without sharing any private data:
\begin{align}
  Z_k = ~ 	& \maximize && \hspace*{-25mm}\vr_k\tr \vx_k,    &&&\label{eqn:ind_model_obj}\\
           	& \subto    && \hspace*{-25mm}\mA_k\vx_k \leq \vs_k,  &&&\label{eqn:ind_c_dp} \\
           	&           && \hspace*{-25mm}\mB_k\vx_k \leq  \vc_k,  &&& \label{eqn:ind_ck_dp}\\
           	&           && \hspace*{-25mm}\zv \leq \vx_k \leq \ov.   &&& \label{eqn:ind_model_nonneg}
\end{align}
However, this approach has three drawbacks: First, it is not clear how
to determine the optimal partitioning of the shared capacity among the
parties without having the complete information of the network
problem. As a result of this suboptimal partitioning, the capacities
allocated to a party might be left unused, even though those
capacities could have been filled up by the other parties, if they had
been shared. These adverse effects of the defragmentation of network
capacities have also been observed in the literature \citep{Curry90, kunnumkal10}. This suboptimal partitioning leads us to the second
drawback. The pre-allocation of the common capacities yields less
total expected revenue than that of the collective model
\eqref{eqn:org_model_obj}-\eqref{eqn:org_model_nonneg}. In other
words, irrespective of the way in which the common capacity is shared,
we have $\sum_{k \in \CK} Z_k \leq Z$. This result simply follows from
the fact that the collection of feasible solutions to each
\eqref{eqn:ind_model_obj}-\eqref{eqn:ind_model_nonneg} is also a
feasible solution to
\eqref{eqn:org_model_obj}--\eqref{eqn:org_model_nonneg}. Third, the
primal and dual variables obtained from the individual models depend on the
partitioning of the common capacity and lack information about the
entire network. Overall, it is more beneficial for all parties to
collaborate and solve the collective model
\eqref{eqn:org_model_obj}--\eqref{eqn:org_model_nonneg}.

\subsection{Data-Private Mathematical Model.}
\label{subsec:DataPrivate}

Even though they may have agreed to collaborate, the major concern for
the parties is the privacy of the input data (see Definition \ref{def:dataset}) and their sensitive
decisions when solving the joint problem
\eqref{eqn:org_model_obj}-\eqref{eqn:org_model_nonneg}. In our problem setup, the parties keep their revenue and capacity vectors as well as their product matrices private. Although the parties sign up for sharing some capacities on the network, they do not share any information about their individual capacities. Thus, only the shared resource capacities are not private for the involved parties. In the subsequent part of this section, we
present the steps for the parties to randomly transform their private
input and output data in order to ensure data-privacy while collectively solving the network problem. Then, with this transformed data, the overall
private model is constructed and made available to all parties. We
conclude this section with our key result, which shows that parties
can still recover their optimal allocations and bid-prices after the
proposed private model is solved by each party. In our following discussion, we use the term \textit{masked problem} to refer to the resulting problem after applying random transformations to the input data.
  
First, we start with concealing the private output; that is, the
individual optimal capacity allotments ($\vx_k^*$) and the individual
bid-prices ($\valpha_k^*$). To this end, we first ask each party
$k \in \CK$ to generate its own private pair of random vectors,
$\veta_k \in \RR^{n_k}$ and $\vxi_k \in \RR^{m_k}$ to transform the primal and dual solutions. Then, we use the auxiliary
variables $\vz_k$ and $\vv_k$ for $k \in \CK$ to construct the
following mathematical model:
\begin{align}
  \maximize & \sum_{k \in \CK}(\vr_k + \mB_k\tr\vxi_k)\tr\vz_k + \sum_{k \in \CK} \vxi_k\tr\vv_k  \label{eqn:org_obj_dp} & \\
  \subto    & \sum_{k \in \CK}\mA_k\vz_k \leq \vc + \sum_{k \in \CK}\mA_k\veta_k, && & & \hspp(\vbeta) \label{eqn:org_c_dp}\\
            & \mB_k\vz_k + \vv_k = \vc_k + \mB_k\veta_k, && k \in \CK, & & \hspp(\vbeta_k) \label{eqn:org_ck_dp}\\
            & \vz_k \leq \ov + \veta_k, && k \in \CK, & &  \label{eqn:org_ub_dp} \\
            & \vz_k \geq \veta_k, && k \in \CK, & & \label{eqn:org_lb_dp} \\
            & \vv_k \geq \zv, && k \in \CK, \label{eqn:org_lastcon_dp}
\end{align}
where the vectors in parentheses are again the dual vectors associated
with the corresponding constraints. The constraints
(\ref{eqn:org_c_dp}) and (\ref{eqn:org_ck_dp}) correspond to the
capacity constraints (\ref{eqn:org_model_c}) and
(\ref{eqn:org_model_ck}) in path-based formulation
\eqref{eqn:org_model_obj}--\eqref{eqn:org_model_nonneg},
respectively. Note that, we use the transformations
$\vz_k = \vx_k + \veta_k$ and $\vv_k = \vc_k - \mB_k\vx_k$ for
$k \in \CK$ to obtain the new model. Since $\vx_k$ is increased by
$\veta_k$, the right hand side of constraints (\ref{eqn:org_model_c})
and (\ref{eqn:org_model_ck}) are increased by
$\sum_{k \in \CK}\mA_k\veta_k$ and $\mB_k\veta_k$ in constraints
\eqref{eqn:org_c_dp} and \eqref{eqn:org_ck_dp}, respectively. Given
that the auxiliary variable $\vv_k$ corresponds to the slack of
constraints (\ref{eqn:org_model_c}), we can rewrite the constraint
\eqref{eqn:org_ck_dp} as an equality. The new auxiliary variable
$\vv_k$ and its cost coefficient $\vxi_k$ are introduced to make sure
that the dual optimal solution is shifted with a random vector. The
following lemma formally shows that the optimal allocations and the
dual variables for each party are indeed perturbed with private random
noise vectors after solving this problem. As long as each party
$k \in \CK$ does not share its random vectors $\veta_k$ and $\vxi_k$
with the other parties, the individual optimal capacity allotments and
the individual bid-prices remains private. The proof of the lemma is
given in the Appendix \ref{sec:app2}.

\begin{lemma}
  \label{lem:output}
  If we denote the primal optimal solution of
  \eqref{eqn:org_obj_dp}-\eqref{eqn:org_lastcon_dp} by
  $(\vz^*_k, \vv^*_k)_{k \in \CK}$ and the dual optimal variables
  associated with the capacity constrains by
  $(\vbeta^*, \vbeta^*_k)_{k \in \CK}$, then we have
  \[
    \begin{array}{cl}
    \vz_k^* = \vx_k^* + \veta_k, & k \in \CK, \\
      \vbeta^* = \valpha^*, &\\
      \vbeta_k^* = \valpha^*_k + \vxi_k, & k \in \CK.
    \end{array}
  \]
\end{lemma}

 We note that, in order to solve problem
  \eqref{eqn:org_obj_dp}-\eqref{eqn:org_lastcon_dp}, each party
  $k \in \CK$ still needs to reveal its pair of random vectors,
  $\veta_k$ and $\vxi_k$ so that the objective function and the bound
  constraints can be constructed. Consequently, the optimal
  allocations and the dual variables of each party are no longer
  private. Thus, our next step is to conceal the random vectors by
  using a linear transformation. That is, we set
  $\vv_k = \mE_k\tr \vw_k$ for $k \in \CK$, where $\mE_k$ is a
  $t_k \times m_k$ random matrix with $t_k \geq m_k$. Likewise, we can
  also set $\vz_k = \mD_k\tr \vu_k$ for $k \in \CK$, where $\mD_k$ is
  a $s_k \times n_k$ random matrix with $s_k \geq n_k$. This is, in
  fact, the transformation proposed by
  \cite{mangasarian2011privacy}. We note that we can simply form
    matrices $\mathbf{D}_k$ and $\mathbf{E}_k$ with real or rational
    random values. Then, the resulting matrices are almost-surely full
    rank \citep{Feng07}. We then obtain
\begin{align}
  \maximize & \sum_{k \in \CK}(\vr_k + \mB_k\tr\vxi_k)\tr\mD_k\tr\vu_k + \sum_{k \in \CK} \vxi_k\tr\mE_k\tr\vw_k & \label{eqn:org_obj_mask1}\\
  \subto    & \sum_{k \in \CK}\mA_k\mD_k\tr\vu_k \leq \vc + \sum_{k \in \CK}\mA_k\veta_k, & \label{eqn:org_mask1_c} \\
            & \mB_k\mD_k\tr\vu_k + \mE_k\tr\vw_k = \vc_k + \mB_k\veta_k, & k \in \CK, \label{eqn:org_mask1_ck}\\
            & \mD_k\tr\vu_k \leq \ov + \veta_k, & k \in \CK, \label{eqn:org_mask1_bndk1} \\
            &  \mD_k\tr\vu_k \geq \veta_k, & k \in \CK, \label{eqn:org_mask1_bndk2} \\
            & \mE_k\tr\vw_k \geq \zv, & k \in \CK. \label{eqn:org_mask1_nonk}
\end{align}

Nonetheless, this transformation is still not enough to conceal the
data or the random vectors because the parties have to explicitly
share the random matrices $\mE_k$ due to constraints
\eqref{eqn:org_mask1_ck} and \eqref{eqn:org_mask1_nonk}. Likewise, the
random matrices $\mD_k$ as well as the random vectors $\veta_k$ need
to be revealed because of the bound constraints
\eqref{eqn:org_mask1_bndk1}-\eqref{eqn:org_mask1_bndk2}. In fact,
\cite{mangasarian2011privacy} deals only with linear programming
models \textit{without} bound constraints and mentions that there is
``a difficulty associated with possibly including non-negativity
constraints.'' \cite{li2013privacy} include inequality constraints, and
resolve this privacy issue by allowing each party to generate a
positive diagonal random matrix for their slack variables. Their approach has been shown to be open to attacks \citep{hong2014inference}. We, on the
other hand, propose to sample from the set of $M$-matrices for which
the positive diagonal matrices constitute a subset. This choice is
valid because if $\mS$ is an $M$-matrix, then $\mS \vx \geq \zv$
implies $\vx \geq \zv$ \citep{Horn91}. This leads to the following
model:
\begin{align}
  \maximize & \sum_{k \in \CK}(\vr_k + \mB_k\tr\vxi_k)\tr\mD_k\tr\vu_k + \sum_{k \in \CK} \vxi_k\tr\mE_k\tr\vw_k & \label{eqn:org_obj_mask2}\\
  \subto    & \sum_{k \in \CK}\mA_k\mD_k\tr\vu_k \leq \vc + \sum_{k \in \CK}\mA_k\veta_k, & \label{eqn:org_mask2_c} \\
            & \mF_k\mB_k\mD_k\tr\vu_k + \mF_k\mE_k\tr\vw_k = \mF_k(\vc_k + \mB_k\veta_k), & k \in \CK, \label{eqn:org_mask2_ck}\\
            & \mG_k\mD_k\tr\vu_k \leq \mG_k(\ov + \veta_k), & k \in \CK, \label{eqn:org_mask2_bndk1} \\
            & \mH_k\mD_k\tr\vu_k \geq \mH_k\veta_k, & k \in \CK, \label{eqn:org_mask2_bndk2}\\
            & \mL_k\mE_k\tr\vw_k \geq \zv, & k \in \CK, \label{eqn:org_mask2_nonk}
\end{align}
where the $m_k \times m_k$ matrices $\mF_k$ and $\mL_k$ as well as the
$n_k \times n_k$ matrices $\mG_k$ and $\mH_k$ are all $M$-matrices. In
order to conceal the random vector $\eta_k$ and the random matrices
$\mD_k$ and $\mE_k$ for $k\in \mathcal{K}$, we multiply constraints
(\ref{eqn:org_mask1_ck})-(\ref{eqn:org_mask1_nonk}) by $M$-matrices
$\mF_k$, $\mG_k$, $\mH_k$ and $\mL_k$, respectively and obtain
constraints (\ref{eqn:org_mask2_ck}) -
(\ref{eqn:org_mask2_nonk}). Model (\ref{eqn:org_obj_mask2}) -
(\ref{eqn:org_mask2_nonk}) ensures privacy of individual input and
output data of each party while parties solve their joint capacity
control problem. Next, we give the formal definition of $M$-matrix
that is used to conceal data in our data-private model
(\ref{eqn:org_obj_mask2}) - (\ref{eqn:org_mask2_nonk}).
  \begin{definition}[$M$-matrix
    \citep{poole1974survey}]
    \label{defn:mMatrix}
    An $\ell\times \ell$ matrix $\mM$ that can be expressed in the
    form $\mM = s\mI - \mN$, where $\mN=(n_{ij})$ with
    $n_{ij}\geq 0,\ i,j \in 1,...,\ell$, and $s > \rho(\mN)$ is called
    an \textit{$M$-matrix} where
    $\rho(\mN)=\max \{|\lambda|: \det(\lambda \mI - \mN) = 0\}$.
  \end{definition}
  This definition also gives a procedure to obtain a random
  $M$-matrix: First sample a random nonnegative $\mN$ matrix and
  select a random $s > \rho(\mN)$. Then, $s\mI - \mN$ becomes an
  $M$-matrix. This simple procedure clearly shows that it is possible
  to produce infinitely many $M$-matrices. We will make use of this
  observation, when we discuss the security of our transformed problem
  in the next section.

To simplify our notation, we further define for $k \in \CK$ the
following
\begin{equation}
  \label{eqn:notation}
\begin{array}{cccc}
  \bar{\vr}_k = \mD_k(\vr_k + \mB_k\tr\vxi_k), & \bar{\vxi}_k = \mE_k\vxi_k, & \bar{\mA}_k = \mA_k\mD_k\tr, & \bar{\vc}=\vc + \sum_{k \in \CK}\mA_k\veta_k = \vc + \sum_{k \in \CK} \tilde{\veta}_k\\
  \bar{\mB}_k = \mF_k\mB_k\mD_k\tr, & \bar{\mF}_k=\mF_k\mE_k\tr & \bar{\vc}_k = \mF_k(\vc_k + \mB_k\veta_k), & \bar{\mG}_k = \mG_k\mD_k\tr \\
  \bar{\ov}_k=\mG_k(\ov + \veta_k) &  \bar{\mH}_k=\mH_k\mD_k\tr, & \bar{\veta}_k = \mH_k\veta_k, & \bar{\mL}_k= \mL_k\mE_k\tr,
\end{array}
\end{equation}
and rewrite model \eqref{eqn:org_obj_mask2}-\eqref{eqn:org_mask2_nonk} as

\begin{align}
  && \bar{Z} = ~ & \maximize &&\hspace*{-10mm} \sum_{k \in \CK}\bar{\vr}_k\tr\vu_k + \sum_{k \in \CK}\bar{\vxi}_k\tr\vw_k & \label{eqn:pri_obj}\\
  &&              & \subto    &&\hspace*{-10mm} \sum_{k \in \CK}\bar{\mA}_k\vu_k \leq \bar{\vc}, && && && \hspp (\vgamma) \\
  &&              &          &&\hspace*{-10mm} \bar{\mB}_k\vu_k + \bar{\mF}_k\vw_k = \bar{\vc}_k, && k \in \CK, && && \hspp (\vgamma_k)\\
  &&              &          &&\hspace*{-10mm} \bar{\mG}_k\vu_k \leq \bar{\ov}_k, && k \in \CK, \\
  &&              &          &&\hspace*{-10mm} \bar{\mH}_k\vu_k \geq \bar{\veta}_k, && k \in \CK, \\
  &&              &          &&\hspace*{-10mm} \bar{\mL}_k\vw_k \geq \zv, && k \in \CK, \label{eqn:pri_nonneg}
\end{align}
where $(\vgamma, \vgamma_k)_{k \in \CK}$ are the dual variables. The
following theorem shows that after the random transformations the
\textit{exact} primal and dual solutions of the original problem can easily be
recovered. The proof of this theorem is given in the appendix.

\begin{theorem}
  \label{thm:equiv}
  Let $(\vu_k^*, \vw_k^*)_{k \in \CK}$ and
  $(\vgamma^*, \vgamma_k^*)_{k \in \CK}$ be the primal and dual
  optimal solutions of
  \eqref{eqn:pri_obj}-\eqref{eqn:pri_nonneg}. Using again the primal
  and dual optimal solutions, $(\vx_k^*)_{k \in \CK}$ and
  $(\valpha^*, \valpha^*_k)_{k \in \CK}$ of the original problem
  \eqref{eqn:org_model_obj}--\eqref{eqn:org_model_nonneg}, we obtain
  \[
    \begin{array}{cl}
    Z = \bar{Z} - \sum_{k \in \CK} \vr_k\tr \veta_k - \sum_{k \in \CK} (\vc_k + \mB_k\veta_k)\tr\vxi_k, &\\
      \vx_k^* = \mD_k\tr \vu_k^* - \veta_k, & k \in \CK, \\
      \valpha^* = \vgamma^*, & \\
      \valpha_k^* = \mF_k\tr\vgamma^*_k - \vxi_k, & k \in \CK.
    \end{array}
  \]
\end{theorem}

With this main theorem, we conclude that the parties can safely obtain
their own eaxact solutions, since the set of random matrices designated with
subscript $k$ is known only to the individual party $k \in \CK$. It is
important to note that the parties can generate their primal solutions
by using the random matrices, yet the dual solutions of the original
problem are exactly the same as the transformed problem
($\vgamma^* = \valpha^*$).

Algorithm \ref{alg:fw} presents our transformation-based protocol and shows how one party ($\hat{k}$ in the algorithm) can apply the data-private capacity control. In Step
\ref{alg:s1}, the party prepares the input data. This data is
transformed in Step \ref{alg:s2}, and shared with the other
parties. Now, the input for the overall private model is available to
everyone (Step \ref{alg:s3}). Each party then solves the private
problem and obtains the optimal solutions in Step \ref{alg:s4}. Then,
party $\hat{k}$ recovers its optimal dual variable vectors in Step
\ref{alg:s5}.  As a result, using Lemma \ref{lem:output} and Theorem
\ref{thm:equiv}, we show that the optimal primal and dual solutions of
the original model can be obtained safely from the optimal primal and
dual solutions of the transformed model. Hence, we conclude that the
correctness of the original problem is preserved.

\begin{algorithm}
	\caption{Data-Private Capacity Control for Party ${\hat{k}} \in \CK$} \label{alg:fw}
	\begin{algorithmic}[1]
		\State Compile private individual input \label{alg:s1}
		\[
		\vxi_{\hat{k}}, \veta_{\hat{k}},\mD_{\hat{k}}, \mE_{\hat{k}}, \mF_{\hat{k}}, \mG_{\hat{k}}, \mH_{\hat{k}}, \mL_{\hat{k}}.
		\]
		\State Transform individual input using \eqref{eqn:notation} and share \label{alg:s2}
		\[
		\bar{\vr}_{\hat{k}}, \bar{\vxi}_{\hat{k}}, \bar{\mA}_{\hat{k}}, \bar{\mB}_{\hat{k}},
		\bar{\mF}_{\hat{k}}, \bar{\vc}_{\hat{k}}, \bar{\mG}_{\hat{k}}, \bar{\ov}_{\hat{k}}, \bar{\mH}_{\hat{k}},
		\bar{\veta}_{\hat{k}}, \bar{\mL}_{\hat{k}}, \tilde{\veta}_{\hat{k}}.
		\]
		\State Store all transformed data \label{alg:s3}
		\[
		(\bar{\vr}_{k}, \bar{\vxi}_{k}, \bar{\mA}_{k}, \bar{\mB}_{k},
		\bar{\mF}_{k}, \bar{\vc}_{k}, \bar{\mG}_{k}, \bar{\ov}_{k}, \bar{\mH}_{k},
		\bar{\veta}_{k}, \bar{\mL}_{k}, \tilde{\veta}_{k})_{k \in \CK}.
		\]
		\State Solve \eqref{eqn:pri_obj}-\eqref{eqn:pri_nonneg}
		with \label{alg:s4} $\bar{\vc}=\vc + \sum_{k \in \CK}\mA_k\veta_k$
		and the stored data. Obtain transformed optimal solution
		\[
		(\vu_k^*, \vw_k^*, \vgamma^*, \vgamma_k^*)_{k \in \CK}.
		\]
		\State Recover private individual output using the transformed solution
		and the private input in Step \ref{alg:s1}:  \label{alg:s5}
		\[
		\vx_{\hat{k}}^* = \mD_{\hat{k}}\tr \vu_{\hat{k}}^* -
		\veta_{\hat{k}}, ~\valpha^* = \vgamma^*, ~\valpha_{\hat{k}}^* =
		\mF_{\hat{k}}\tr\vgamma^*_{\hat{k}} - \vxi_{\hat{k}}.
		\]
	\end{algorithmic}
\end{algorithm}

\subsection{Security.}
\label{subsec:security}

We next discuss the security of our data private model
\eqref{eqn:pri_obj}-\eqref{eqn:pri_nonneg} in the presence of attacks. With transformation-based approaches, there is an overarching
trade-off between efficiency and security
\citep{goldreich2009foundations}. Actually,
\citet{laud2013possibility} show that it is impossible to achieve an
information theoretical security with the transformation techniques of
multiplication, scaling, permutation and shifting; a point that is
also noted by \cite{dreier2011practical} in their work on linear
programming. In transformation-based methods, the concern is how much information may leak to other parties during the transformation. Indeed, this concern is also raised by \cite{du2002practical} as to quantify the security achieved in each transformation-based protocol, so that these protocols can be compared in terms of the level of security they can achieve. To be able to attain this, the notion of information leakage, also known as the vulnerability of the system, is introduced by \cite{braun2009quantitative} in which it is defined as ``the amount of information learnt by the adversary by observing the output of the protocol.'' In a more recent study, \cite{dreier2011practical} use information leakage concept to measure how much information about the private information is revealed to an adversary. They quantify this leakage when the parameters of a linear program are masked with random matrices. In their analysis, the components of all data and random matrices are assumed to be nonnegative integers, and the leakage results depend on the largest components of the random matrices. Even though our input matrices are binary, we do not impose upper bounds or integrality requirements on the components of the random vectors or matrices (not necessarily square) used in our transformations. It is crucial to point out that the authors only use \textit{positive monomial matrices} (permutation matrices with nonnegative pivot elements) to deal with the inequalities, and state that these matrices cause vulnerability in security. Unlike these simple
permutation matrices, we sample from the much larger set of
$M$-matrices. In fact, as Definition \ref{defn:mMatrix} and the subsequent paragraph show, this set is uncountable.

Recall that the matrices $\mA_k$ and $\mB_k$ constitute the products
for each firm $k \in \CK$. In certain applications, one of the parties
may be able to guess the products and the capacities of the other
parties. For instance, in airline revenue management, these products
correspond to different fare class itineraries which may be collected
by web-scraping. As Theorem \ref{thm:equiv} and \eqref{eqn:notation}
show, the security of our transformations relies mainly on private
random matrices $\mD_k$, $\mF_k$, $\veta_k$ and $\vxi_k$. We show in
the next lemma that even all private data (see Definition \ref{def:dataset}) of a firm are
perfectly guessed, a brute-force approach to obtain the private random
matrices is computationally infeasible. The proof of this lemma is given
in the appendix.

\begin{lemma}
  \label{lem:security}
  Suppose for $k\in \CK$ that $1 \leq m < n_k \leq s_k$,
  $1 < m_k \leq t_k$, and both $\mA_k$ and $\mB_k$ have full rank.
  Even if all private data of party $k\in \CK$ are known (see Definition \ref{def:dataset}), then
  finding any one of $\mD_k$, $\mF_k$, $\veta_k$ or $\vxi_k$ requires
  obtaining a particular solution to a system of linear equations with
  infinitely many solutions.
\end{lemma}

The condition in Lemma \ref{lem:security} implies that each
party should participate in this collaboration with multiple
individual capacities and multiple products. This is a reasonable
assumption, since it is easier for the other parties to guess the
actual values for very small datasets. We note that even if the
dataset of a party is small, the same party can still enlarge its
dataset by adding redundant constraints or dummy products.

\section{Simulation Study.}
\label{sec:simstudy}

We devote this section to our simulation study for discussing
different aspects of our proposed data-private model. In particular,
we investigate the impact of capacity sharing and evaluate the
computational performance of the data-private model. We next explain
our simulation setup in detail and then present our numerical
results.

\subsection{Setup.}
\label{sec:exp_setup}

We design our experiments by using an airline network structure
obtained from an actual firm. These data include flight legs with
corresponding capacities, flight itineraries and origin-destination
(OD) paths. Since the network data belongs to a single airline, it does
not include any alliance information. In order to construct an
alliance network, we randomly allocate OD-paths in each network to
obtain artificially generated airline partners. The partners set a
block space partnership agreement to share capacities on some of the
flights at the beginning of the planning horizon. Although the
real-time flight information such as marketed flight itineraries and
associated prices can be partially available through online travel
agencies during the sale season, the complete flight information
including forecasted demand, prices and flight capacities are not
available when the codeshare agreements are set at the beginning of
the sale season.

We simulate the arrival of reservation requests over a planning
horizon of length $T$. We assume that the booking requests for OD path
$s \in \mathcal{S}$ arrive according to a homogeneous Poisson process
with rate $\lambda_s$. Given that a booking request arrives for
OD-path $s$ at time period $t$, it is for product $i$ with probability
$p_{is}(t)$. The way we generate these arrival probabilities is quite
similar to the one given by \cite{birbil14}. The simulation process is
defined as follows. We first generate the arrival times of booking
requests for all OD-paths over the planning horizon $T$. By using the
arrival probabilities for products in each OD path, we find the
product of the requests and apply the corresponding booking
policies. To change the tightness of the flight capacities, we use a
load factor parameter ($\rho$). The average arrival rate, $\mu_j$ for
flight $j \in \mathcal{J}$ depends on the value of the load
factor. This relation can be expressed as
$\mu_j = \rho \frac{c_j}{TN_j}$, where $N_j$ is the number of OD-paths
using flight leg $j$. Then, the arrival rate $\lambda_s$ for OD-path
$s$ is generated as follows
$\lambda_s = \frac{\sum_{j \in J_s} \mu_j}{J_s}$, where $J_s$ is the
number of flight legs used by OD-path $s$.

Our experimental design is based on various factors. These are the
number of alliance partners ($K$), the number of OD-paths ($N$) and
the load factor ($\rho$). In simulation experiments, we design three
alliance partnerships with different numbers of partners
$K \in \{2, 4, 6\}$. We assume that all alliance partners have a
similar market share in terms of number of OD-paths in the alliance
network. We test three networks with sizes $N \in \{100, 200,
400\}$. We extract these networks from the overall network data, which
include 119, 215 and 368 flight legs, and 869, 1,762 and 3,567
products, respectively. Since we randomly divide OD paths among the
fictitious alliance partners, the number of shared flights can change
depending on the allocated flights. Therefore, Table \ref{table:data2}
presents the average number of shared flights in each subnetwork. The
last parameter set comes from the load factor $\rho \in \{1.2, 1.6\}$
corresponding to medium and high loads, respectively. The
computational results are reported over 100 simulation runs. We take
the reservation period length as $T = 1,000$.

\begin{table}[ht]
	\begin{center}
          \caption{The average number of shared flights in each network}
		\label{table:data2}
		\begin{tabular}{|c|ccc|}
			\hline
			\multicolumn{1}{|c|}{\textbf{Number of Parties}} & \multicolumn{3}{c|}{\textbf{Network Size} $\mathbf{(N)}$}  \\
			$\mathbf{(K)}$& \textbf{100} & \textbf{200}   &  \textbf{400}  \\
			\hline
			\textbf{2} &  9 & 32 & 81  \\
			\textbf{4} & 16 & 42 & 109 \\
			\textbf{6} & 18 & 43 & 124 \\
			\hline
		\end{tabular}
	\end{center}
\end{table}

\subsection{Results.}
In this section, we conduct simulation experiments to evaluate the
effects of collaborative capacity sharing and provide a sensitivity
analysis with respect to various parameters. In particular, we investigate centralized coordination with complete information sharing, coordination with data privacy and individual control strategies. These strategies are formally introduced as follows:

\paragraph{Collaborative Capacity Planning (CP).} This strategy
assumes that alliance partners act collaboratively and the booking
decisions for shared capacities are controlled through an integrated
planning system which requires complete information sharing. \cite{topaloglu12} describes this system as
``centralized planning'' where the booking decisions are made by
considering the overall alliance benefit. CP solves the model
(\ref{eqn:org_model_obj})-(\ref{eqn:org_model_nonneg}) to compute the
optimal values of the dual variables ($\valpha \in \RR^m$ and
$\valpha_k \in \RR^{m_k}$, $k \in \CK$) associated with the capacity
constraints. When a request for a product arrives, the summation of
the optimal dual variables corresponding to the used flight legs
becomes the bid-price for accepting or rejecting the request. That is,
assuming a product request using path $s$ arrives for party $k$, we
accept this request if the fare of the product is greater than or
equal to
$\sum_{j\in \mathcal{J}} a_{js}\alpha_j +\sum_{j\in \mathcal{J}_k}
b_{js} \alpha_{jk}$. To consider the effects of reoptimization, we
divide the planning horizon into five equal segments and resolve the
model (\ref{eqn:org_model_obj})-(\ref{eqn:org_model_nonneg}) at the
beginning of each segment with the updated capacities. Since CP
coordinates the booking decisions for the whole alliance, this
strategy requires access to all flight information of the partners
(the capacities on all flight legs and the expected demands for all
products) which is quite unlikely to occur in practice
\citep{topaloglu12}.

\paragraph{Coordinated Capacity Sharing (CCS).} In this strategy, the
parties come together and solve the data-private model
(\ref{eqn:pri_obj})-(\ref{eqn:pri_nonneg}) by transforming their
private information to obtain the optimal values of the transformed
dual variables and the capacity allocations (see Algorithm
\ref{alg:fw}). After partners receive the transformed solution, each
of them converts the transformed values to the original ones as shown
in Theorem \ref{thm:equiv}. During the booking horizon, each partner
makes its own booking control decisions by using the optimal dual
variables and the allocated leg capacities. Letting,
$(\valpha^*, \valpha^*_k)_{k \in \CK}$ be the recovered dual variables
obtained by the dual optimal solution of model
\eqref{eqn:pri_obj}-\eqref{eqn:pri_nonneg}, we accept the arriving
path $s$ request if the fare of the product is greater than or equal
to
$\sum_{j\in \mathcal{J}} a_{js}\alpha^*_j +\sum_{j\in \mathcal{J}_k}
b_{js} \alpha^*_{jk}$ and there is enough allocated leg capacity for
the flights covered by path $s$.  Similar to CP, we divide the booking
horizon into five segments and resolve problem
(\ref{eqn:pri_obj})-(\ref{eqn:pri_nonneg}) at the beginning of each
segment. Unlike CP, CCS does not require alliance partners to share
any private information regarding flights.

\paragraph{Individual Control (IC).} This strategy solves problem
(\ref{eqn:ind_model_obj})-(\ref{eqn:ind_model_nonneg}) for each
partner. For shared flight-legs, partner-based capacity allocations
are calculated with respect to the expected demands. In particular,
letting $d_{jk}$ be the demand for partner $k$ in shared flight leg
$j$, the allocated capacity for partner $k$ is calculated as
$\frac{d_{jk}}{\sum_{k \in K_j}d_{jk}}c_j$, where $K_j$ is the set of
partners using leg $j$. In this strategy, each partner makes its own
booking control decisions by using the optimal bid-prices associated
with capacity constraints in problem
(\ref{eqn:ind_model_obj})-(\ref{eqn:ind_model_nonneg}). Similar to
previous strategies, we divide the planning horizon into five segments
and revise the bid-prices at the beginning of each segment. IC
strategy requires alliance partners to share their demand information
in order to allocate the capacities of the shared flights.

Recall that the objective function value in path-based formulation can
also be obtained by solving different single-capacity, static and
dynamic programming models. Our approach here is applicable in all
those cases. In our numerical experiments, we assume that the three
strategies listed above use a deterministic linear programming (DLP)
model to compute the booking control policies. The reason behind this
choice is two-fold: First, DLP models are frequently used in the
literature \citep{poundarikapuram04, Albrecht15}. Second, bid-price control
with a DLP model is a competitive strategy when compared against other
static and dynamic network models \citep{Talluri05}.

Figure \ref{fig:revenue} shows our simulation results in terms of revenues (objective functions) for three
alliance networks with two, four and six partners, respectively. In
these figures, we present the relative differences with respect to the
CP strategy, since it always performs better than the other two
strategies. CP oversees the whole alliance network and makes
accept-reject decisions for arriving reservation requests for all
partners. On the other hand, airlines individually make their booking
control decisions in strategies CCS and IC without sharing any
information over the planning horizon. In Figure \ref{fig:revenue},
the dashed line passing through 100 corresponds to CP, and bar charts
are used to show the relative difference for strategies CCS and
IC.

When we compare the respective performances, we observe that the
average revenues obtained by CCS are very close to those obtained by
CP, especially for the networks with 100 and 200 OD-paths. The average
performance gaps between CP and CCS are only 0.20\%, 0.35\% and 0.85\%
for the problems with 100, 200 and 400 OD-paths, respectively. As
Figure \ref{fig:revenue} illustrates, the performance gaps between CP
and CCS slightly decrease when the load factor is high. We conjecture
that CP and CCS make same booking decisions most of the time since
both of these strategies solve the same centralized model to obtain
their optimal booking policies The only difference is that CCS
allocates shared flight capacities to partners; hence, each airline is
restricted by that limit while making booking control decisions. CP
strategy pools the capacities of the shared legs and does not consider
the individual booking limits. When the arrival intensity is high, CCS
can compensate the revenue loss due to these restrictive booking
limits. We have validated that the relative differences between the
total expected revenues obtained by CCS and CP are statistically
significant at 95\% level in 18 test scenarios. For the relative
differences between the total expected revenues achieved by CCS and
CP, we have failed to reject the null hypothesis in only one of the
test instances. Thus, CCS can deliver similar results with the ideal
case where each party shares its information.

Comparing IC with CP and CCS in Figure \ref{fig:revenue}, we observe
that IC obtains lower expected revenues in all cases. We notice that,
as the number of OD-pairs (network size) and the alliance partners
increases, it performance deteriorates. The expected revenues obtained
with IC can lag on average 7.35\% behind those obtained with
CCS. This striking performance gap between strategies CCS and IC is
due to the management of the shared capacities. While CCS solves the
data-private model (\ref{eqn:pri_obj})-(\ref{eqn:pri_nonneg}) to
obtain booking control variables by considering the whole alliance
network, IC allocates the capacities of the shared flight legs by only
considering the expected demand information of each partner. This
demonstrates the importance of considering overall network information
while making capacity allocation decisions.

\begin{figure}[ht]
 \centerline{
     \includegraphics[scale=0.6]{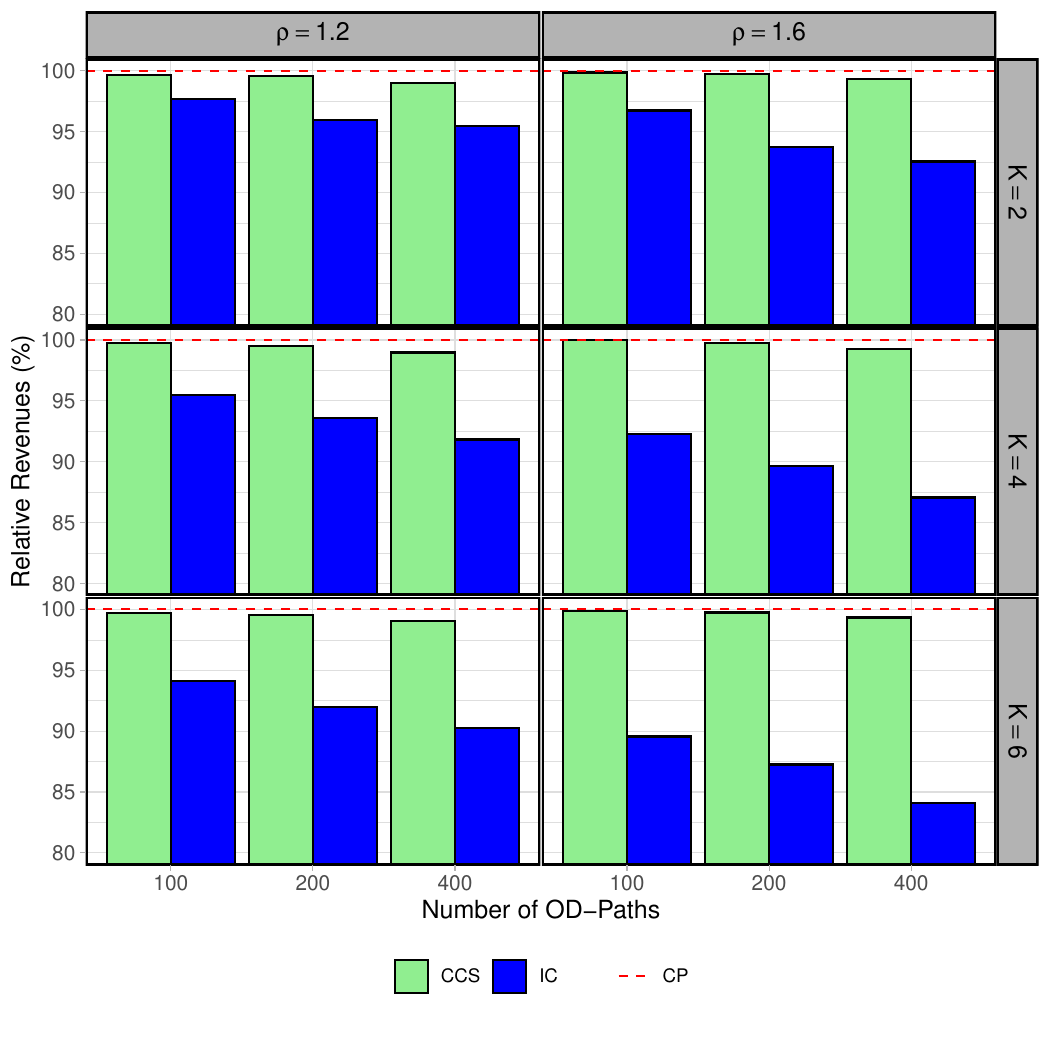}
   }
   \caption{Relative average revenues with respect to CP.}
 \label{fig:revenue}
\end{figure}

\subsection{Computational Efficiency.}
\label{subsec:sparsity}

In the last step, we evaluate the computational efficiency of the
data-private model (\ref{eqn:pri_obj})-(\ref{eqn:pri_nonneg}).
When we consider the structure of the proposed data-private model, we observe that the
matrices in the original model (\ref{eqn:org_model_obj})-(\ref{eqn:org_model_nonneg}) lose their sparse structure after the
reformulation. Take for instance the matrix $\mA_k$ and its transformed
counterpart $\bar{\mA}_k$. An incidence matrix $\mA_k$ is sparse
whereas $\bar{\mA}_k$ is quite dense. This loss of sparsity structure
in the overall problem should be expected to cause an increase in the
computation time. Indeed, we have observed that whenever the matrices
$\bar{\mA}_k$ and $\bar{\mB}_k$ are obtained by straightforward
randomization, then the solution time of the data-private model is
considerably longer than the time to solve the original problem (see
Section \ref{sec:simstudy} for our actual computation times). Figures
\ref{fig:original} and \ref{fig:dense} show the sparsity structure
before and after direct random transformation, respectively.

\begin{figure}
	\centering     %%% not \center
	\subfigure[Original $\mathbf{A}_k$ and $\mathbf{B}_k$ matrix]{
		\label{fig:original}\includegraphics[width=.3\linewidth]{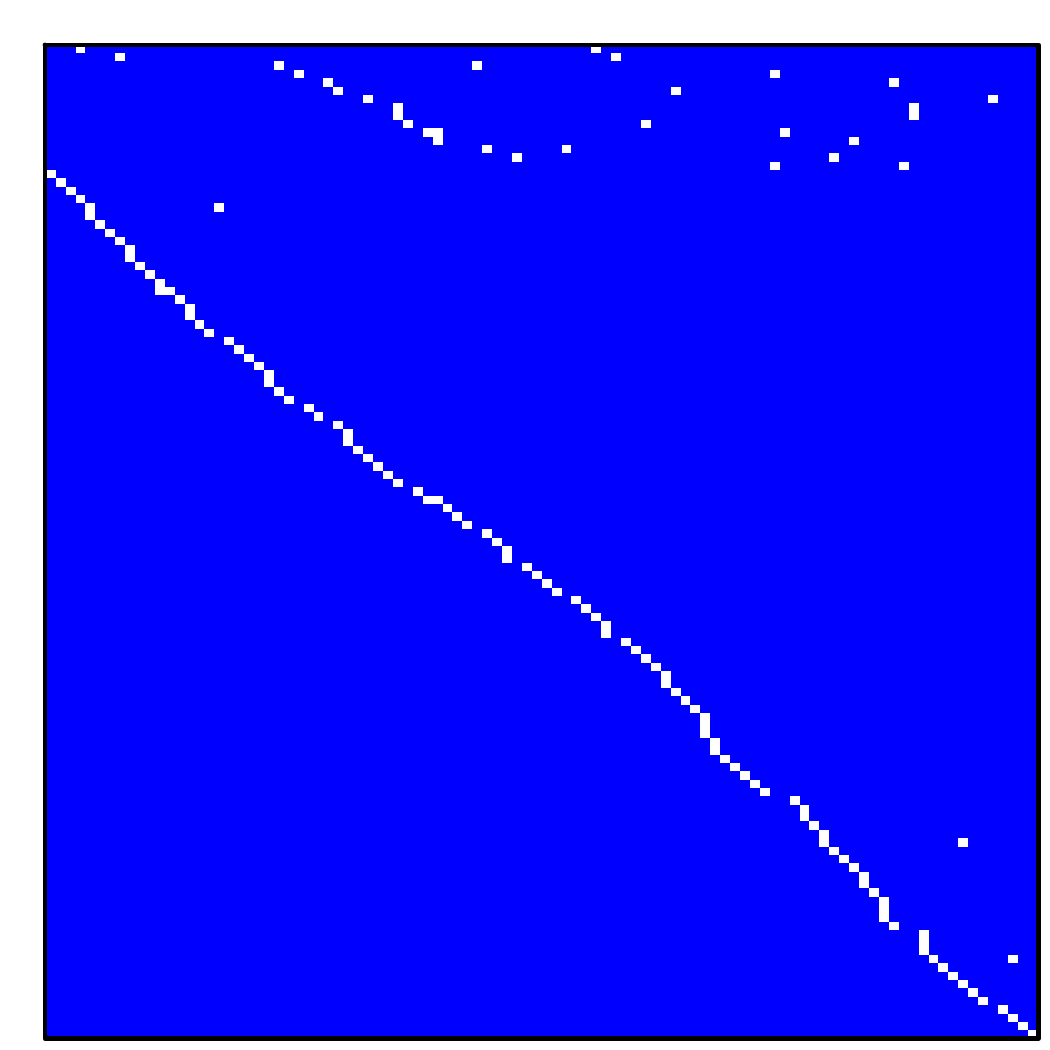}
	}
	\subfigure[Dense $\mathbf{\bar{A}}_k$ and $\mathbf{\bar{B}}_k$ matrix]{
		\label{fig:dense}\includegraphics[width=.3\linewidth]{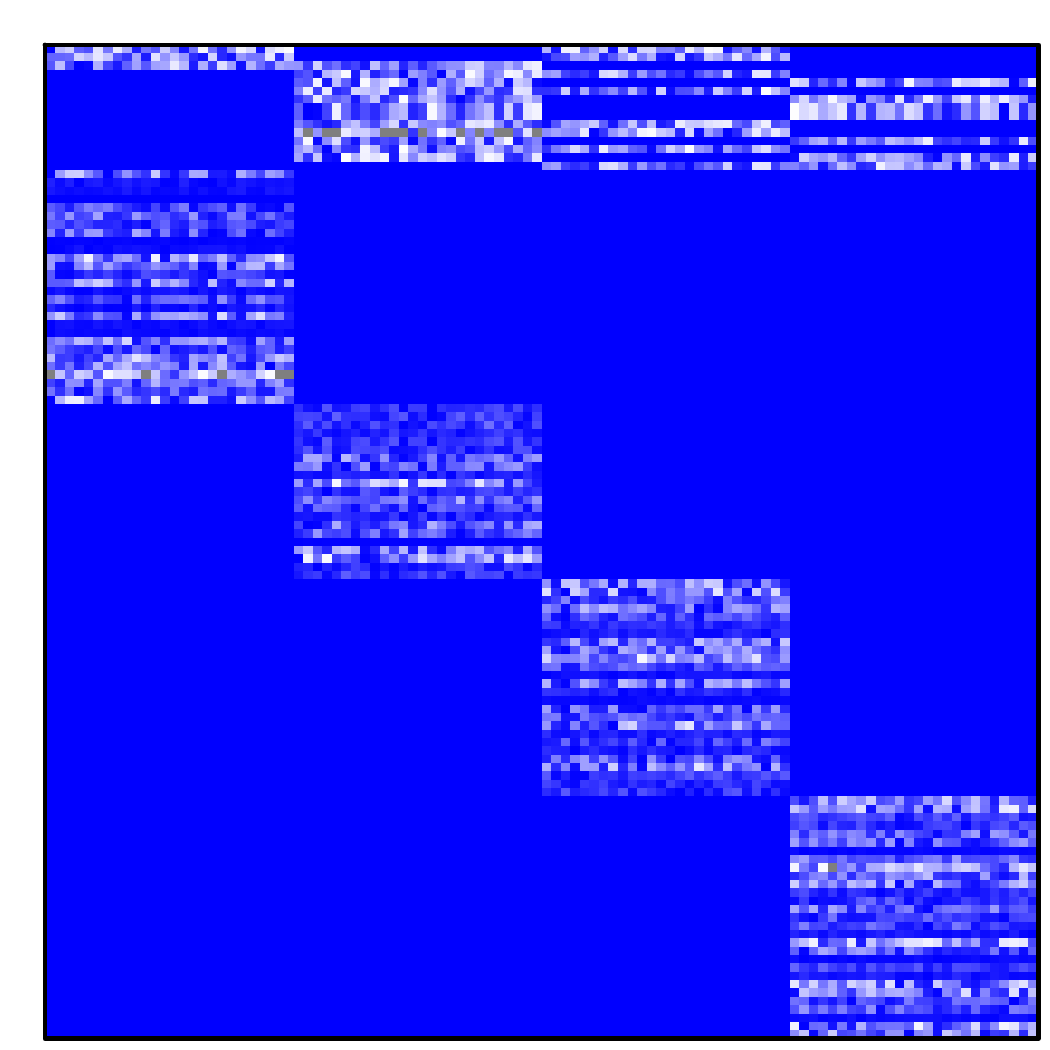}
	}
	\subfigure[Sparse $\mathbf{\bar{A}}_k$ and $\mathbf{\bar{B}}_k$ matrix]{
		\label{fig:sparse}\includegraphics[width=.3\linewidth]{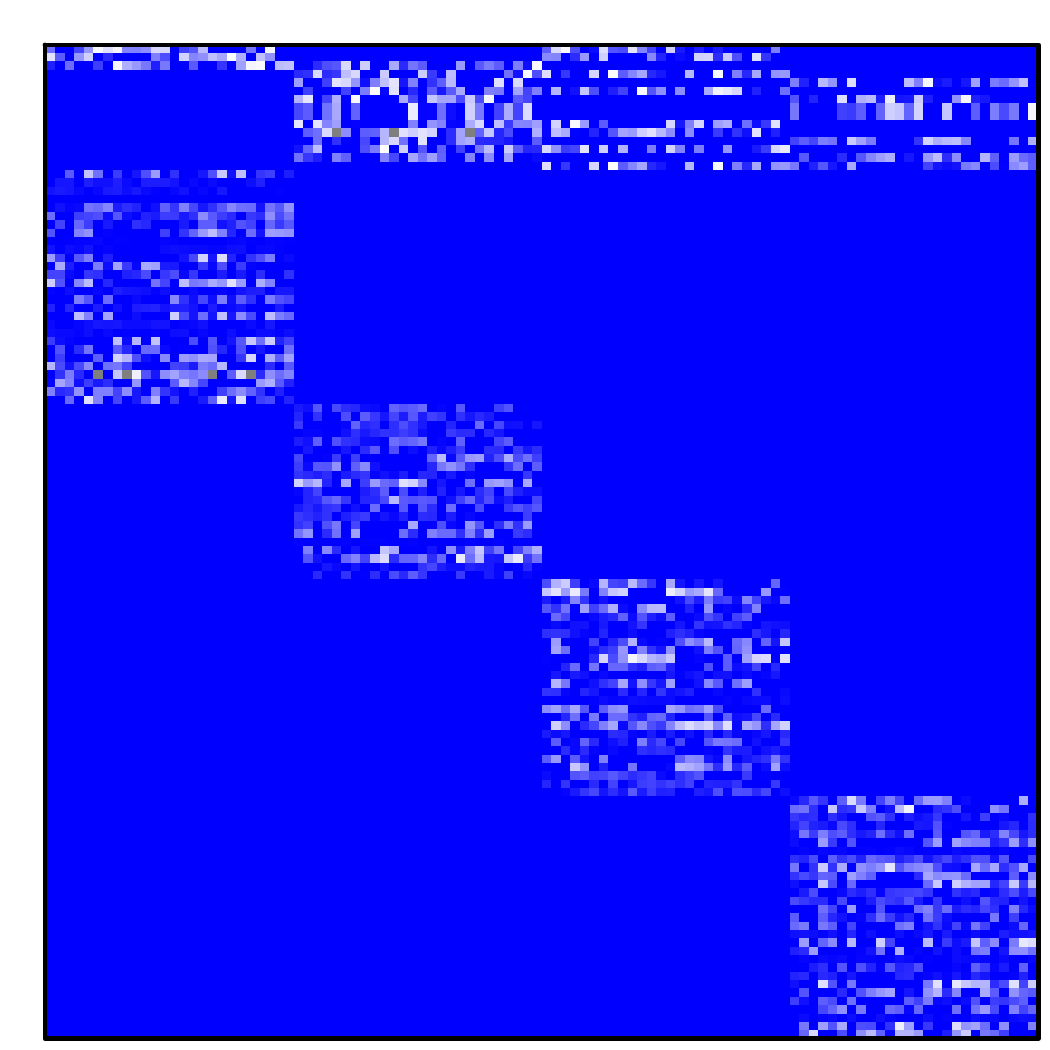}
	}
	% \addtocounter{subfigure}{-1}
	% \subfigure{
	% 	\includegraphics[width=.1\linewidth]{./figures/legend.png}
	% }
	\caption{The sparsity structure of matrices for an example
		problem from our computational study. The values in each
		cell of the matrix is between zero (dark blue) and one
		(bright white). The darker the cell, the closer the value to zero.}
	\label{fig:origvsmaskedsparsity}
\end{figure}

In order to circumvent this loss of sparsity, we try to randomize the
matrices in a structured manner so that we can obtain transformed
matrices that are as sparse as possible. To this end, we aim at
filling in the nonzero entries of the random matrix $\mD_k$ in such a
way that the multiplication of its components with the components of
$\mA_k$ and $\mB_k$ yields as many zeros as possible. This observation
leads to the following mathematical programming model:
\begin{align}
	\minimize &  \ov_m\tr (\mA_k \mU) \ov_{s_k} +  \ov_{m_k}\tr (\mB_k
	\mU) \ov_{s_k} + \ov_{n_k}\tr \mU \ov_{s_k}
	\label{aux:obj} \\
	\subto & \mU\ov_{s_k} \geq s_k \ov_{n_k}, \label{aux:c1} \\
	& \ov_{n_k}\mU \geq s_k \ov_{s_k}, \label{aux:c2} \\
	& \mU \mbox{ is a binary matrix}, \label{aux:nonneg}
\end{align}
where the subscript $\bullet$ in $\ov_\bullet$ shows the dimension of
the vector of ones. The first two terms in \eqref{aux:obj} are added
to obtain as many zeros as possible after multiplying $\mA_k$ and
$\mB_k$ with the binary matrix $\mU$. The last term of \eqref{aux:obj}
makes sure that the solution is filled with zeros instead of ones as
long as the first two terms are not affected. The constraints
\eqref{aux:c1}-\eqref{aux:c2} guarantee that we have $s_k$ many ones
in each column and row of $\mU$. This is a network flow problem
satisfying the total unimodularity property. Therefore, it can be
solved very efficiently by a standard network simplex
algorithm. Moreover, the resulting optimal spanning tree solution
$\mU^*$ has full rank \citep{wright2000minimum}. Then, the last step
is to randomize this binary matrix to obtain the desired
matrix. Formally, $\mD_k\tr = \mU^* \odot \mR$, where $\odot$ stands
for the Hadamard product and $\mR$ is an $n_k \times s_k$ random
matrix. When contrasted against Figure \ref{fig:dense}, Figure
\ref{fig:sparse} shows how obtaining the matrix $\mD_k$ by solving
\eqref{aux:obj}-\eqref{aux:nonneg} changes the sparsity structure of
the data-private model.

To understand the effect of transformation, we report the computation
times for the original model
(\ref{eqn:org_model_obj})-(\ref{eqn:org_model_nonneg}), the
data-private model (\ref{eqn:pri_obj})-(\ref{eqn:pri_nonneg}) and the sparsity induced model. The
data-private model results are first given with straightforward
randomization, which ends with full matrices. Then, we solve
\eqref{aux:obj}-\eqref{aux:nonneg} to obtain sparse matrices.  We
evaluate the computation times for all network sizes with four
parties. Figure \ref{fig:runtimeAvg} presents the average
    computation times on a semi-logarithmic plot; that is, the values
    on the vertical axis are scaled by taking their logarithms. The
legend shows the original model (CP), the masked model with
straightforward randomization (CCS - Dense) and the masked model with
sparsity inducing transformations (CCS - Sparse). Our numerical results confirm that
this loss causes a significant increase in the computation time. The
data-private model with dense matrices takes by far the largest
computation time compared to other models. As the network size
increases, the solution time of the data-private model also
increases. Taking sparsity into consideration for the data-private
model does indeed pay off, as the computation time with the sparsity
inducing transformations reduces the computation times considerably.

\begin{figure}[h]
	\centerline{
		\includegraphics[scale=0.6]{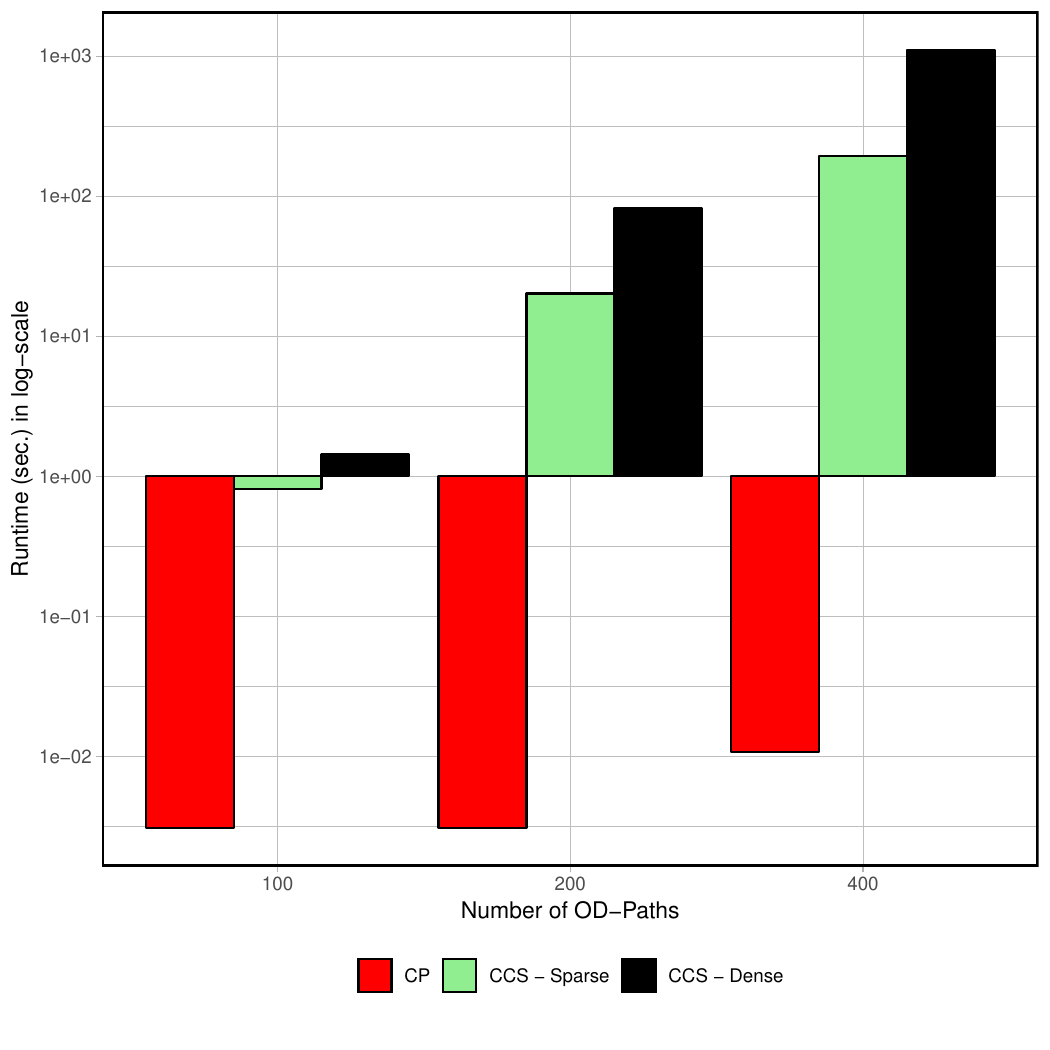}
    }
	\caption{Average computation time ($\rho = 1.2, K = 4$).}
	\label{fig:runtimeAvg}
\end{figure}

At this point, we should emphasize that the random matrices obtained
after solving the mathematical programming model
\eqref{aux:obj}-\eqref{aux:nonneg} may not be secure in the sense of
Section \ref{subsec:security}. Thus, the gain from maintaining
sparsity may come at the cost of a security breach. This happens
because we do not have a control on the optimal solution of the model,
and hence, it is not easy to quantify the potential leakage
\citep{hong2018privacy}. We leave this discussion for a future work.

\section{Conclusion.}
\label{sec:conc}

We have presented a mathematical model which considers data privacy in
collaborative resource management problem when multiple parties
share some of the capacities of the network. The proposed approach is
based on applying matrix transformations to collaborative network
problem. We have shown that
the original primal and dual optimal solutions can be derived from the
proposed data-private mathematical model. We have also discussed the security of the input data after solving the transformed problem.

We have conducted a simulation study on a network structure of an
airline. Our results have illustrated the benefits of the proposed
data-private capacity control. We have considered the setting with
and without a collaboration, and shown that, with collaboration the revenues
for the parties are significantly higher. Therefore, these results
offer an economic motivation for the parties to form an
alliance. Nevertheless, the privacy comes at a cost. Unlike the sparse
structure of the original problem, the data-private model has a dense
structure. This loss of sparsity causes a considerable increase in
computation times. To overcome this problem, we have provided an
approach based on solving a network flow model. We have demonstrated
how this approach positively affects the computational effort. We have
also cautioned that our approach for maintaining the sparse structure
may come at a security leakage cost. An important point that we have
left for future study.

Even though we have conducted a simulation study on an airline
network, our approach is not limited to airline problems. The proposed
approach can be used in different network problems from collaborative supply chains to power networks. For instance, in the logistics sector, sharing network capacities among several companies may increase the utilization of the resources. This high utilization, in
turn, increases the competitive advantage of a participating company
in terms of higher profit and less environmental impact. Furthermore, our approach can be used for decision-making problems outsourced in a cloud environment. Cloud computing provides computing resources and it is widely used by companies to efficiently solve their large-scale decision making problems. However, one of the main issues is the security and the privacy of the stored data. Our approach can be used to mask input data and solve the problem while ensuring data privacy.

There are other interesting research questions about data privacy in
network management. We have presented a transformation-based
approach for capacity control. Although we guarantee to obtain the
exact optimal solutions for each party, our approach may become vulnerable
for potential security breaches especially for small-scale
problems. Investigating the effect of small-scale problems on data-privacy is one of our
future research directions. Another approach to tackle data privacy in collaborative decision making problems could be using the concept of differential-privacy, where the main idea is to perturb the output with a carefully adjusted noise. Such an approach does lead to
approximate solutions but the privacy levels can be quantified and
controlled. This approach is also on our agenda for future research.

\clearpage

\onehalfspacing

\bibliographystyle{apalike}
\bibliography{RM}

\clearpage

\appendix

\section{}
\label{sec:app1}

Suppose for each path or route
$s \in \CS$ that the objective function $\phi_{s}(x_{s})$ consists of $B_s$ breakpoints
(intervals with a length of one) that subdivides the range of
$x_s$. These breakpoints collectively form the set $\CB_s$ for
$s \in \CS$. We introduce the auxiliary variables $x_{b s}$ for
$b \in \CB_s$, and set
\[
x_s = \sum_{b \in \CB_s} x_{b s}.
\]
Since the length of each interval is one, we have
$x_{b s} \leq 1, s \in \CS, ~ b \in \CB_s$. If we denote the partial
revenues by $r_{bs}$, then the new model becomes
\begin{align*}
\maximize & \sum_{s \in \CS}  \sum_{b\in \CB_s} r_{bs}x_{bs}, \\
\subto & \sum_{s \in \CS}  \sum_{b\in \CB_s}  a_{js}x_{bs} \leq c_{j},& j \in \CJ, \\
& \sum_{s \in \CS_k}  \sum_{b\in \CB_s}  a_{js}x_{bs} \leq c_{j},& j \in \CJ_k, k\in\CK, \\
& 0 \leq x_{bs} \leq 1, & s \in \CS, ~ b \in \CB_s.
\end{align*}
Due to the concavity of the objective function, we have
$r_{1s} \geq r_{2s} \geq r_{B_ss}$ for $s \in \CS$. This
structure allows us to partition for $k\in\CK$, the decision variables
and the objective function parameters as
\[
\vx_k = [x_{bs}: s \in \CS_k, b \in \CB_s]\tr \mbox{ and }
\vr_k = [r_{bs} : s \in \CS_k, b \in \CB_s]\tr,
\]
respectively. Again for $k\in\CK$, we next define the $m \times n_k$
matrix $\mA_k$ with $n_k = \sum_{s \in \CS_k} B_s$ and the
$m_k \times n_k$ matrix $\mB_k$ as

\begin{equation}
\label{eq:AkBk}
\mA_k = \left[ \underset{B_s \mbox{ times}}{\underbrace{a_{js} ~
		a_{js} ~ \cdots ~ a_{js}}} \right]_{j \in \CJ, s \in \CS_k}
~\mbox{ and }\hspp
\mB_k = \left[ \underset{B_s \mbox{ times}}{\underbrace{b_{js} ~
		b_{js} ~ \cdots ~ b_{js}}} \right]_{j \in \CJ_k, s \in \CS_k},
\end{equation}
respectively. Here $b_{js} =1$, if path $s$ uses one unit from
capacity $j$; otherwise, $b_{js} =0$. The last step is to introduce
the shared and the private capacity vectors as
\[
\vc = [c_j: j \in \CJ]\tr ~\mbox{ and }~ \vc_k = [c_j: j \in \CJ_k]\tr \mbox{ for all } k \in \CK,
\]
respectively. We are now ready to give our main capacity sharing model with the path-based formulation:
\begin{align*}
Z = ~ & \maximize && \hspace*{-20mm}\sum_{k \in \CK} \vr_k\tr \vx_k,  \\
& \subto    && \hspace*{-20mm}\sum_{k \in \CK} \mA_k\vx_k \leq \vc, &&             && \hspp (\valpha)  \\
&           && \hspace*{-20mm}\mB_k\vx_k \leq  \vc_k,               && k \in \CK,  && \hspp (\valpha_k) \\[2mm]
&           && \hspace*{-20mm}\zv \leq \vx_k \leq \ov, && k \in \CK, &&
\end{align*}
where $\ov$ and $\zv$ stand for the vector of ones and the vector of zeros, respectively.

\section{}
\label{sec:app2}

We have reserved this section for the proofs of our theoretical results,
which we have repeated here for clarity of presentation.

\lemnum{\ref{lem:output}}{
	If we denote the primal optimal solution of
	\eqref{eqn:org_obj_dp}-\eqref{eqn:org_lastcon_dp} by
	$(\vz^*_k, \vv^*_k)_{k \in \CK}$ and the dual optimal variables
	associated with the capacity constrains by
	$(\vbeta^*, \vbeta^*_k)_{k \in \CK}$, then we have
	\[
	\begin{array}{cl}
	\vz_k^* = \vx_k^* + \veta_k, & k \in \CK, \\
	\vbeta^* = \valpha^*, &\\
	\vbeta_k^* = \valpha^*_k + \vxi_k, & k \in \CK.
	\end{array}
	\]
}

\begin{proof}
	We first define $(\vlambda_k)_{k \in \CK}$ as the dual vector
	corresponding to the upper bound constraints
	\eqref{eqn:org_model_nonneg}. Then, the dual of
	\eqref{eqn:org_model_obj}--\eqref{eqn:org_model_nonneg} becomes
	\begin{align}
		\minimize & \vc\tr \valpha +  \dsum_{k \in \CK} \vc_k\tr \valpha_k + \sum_{k\in\CK} \ov \tr \vlambda_k \label{eqn:dual_model_obj} \\
		\subto    & \mA_k\tr\valpha + \mB_k\tr\valpha_k + \vlambda_k \geq \vr_k, & k \in \CK, &&\\
		& \valpha, \valpha_k, \vlambda_k \geq \zv, & k \in \CK. \label{eqn:dual_model_nonneg}
	\end{align}
	Likewise, we also define $(\vnu_k)_{k \in \CK}$ and
	$(\vtheta_k)_{k \in \CK}$ as the dual vectors corresponding to the
	constraints \eqref{eqn:org_ub_dp} and \eqref{eqn:org_lb_dp},
	respectively. Then, the dual of
	\eqref{eqn:org_obj_dp}-\eqref{eqn:org_lastcon_dp} is obtained as
	\begin{align}
		\minimize & (\vc + \dsum_{k \in \CK}\mA_k \veta_k)\tr \vbeta +  \dsum_{k \in \CK}
		(\vc_k + \mB_k\veta_k)\tr \vbeta_k + \sum_{k\in\CK} (\ov +
		\veta_k)\tr \vnu_k - \dsum_{k \in \CK} \veta_k\tr \vtheta_k \label{eqn:dual_obj_dp}\\
		\subto    & \mA_k\tr\vbeta + \mB_k\tr\vbeta_k + \vnu_k - \vtheta_k
		= \vr_k + \mB_k\tr\vxi_k, && k \in \CK, \\
		& \vbeta_k \geq \vxi_k, && k \in \CK,  \\
		& \vbeta, \vnu_k, \vtheta_k \geq \zv, && k \in \CK. \label{eqn:dual_nonneg_dp}
	\end{align}
	Suppose that $(\vz^*_k, \vv^*_k)_{k \in \CK}$ and
	$(\vbeta^*, \vbeta^*_k, \vnu^*_k, \vtheta^*_k)_{k \in \CK}$ are
	the primal and the dual optimal solutions for
	\eqref{eqn:org_obj_dp}--\eqref{eqn:org_lastcon_dp},
	respectively. Let
	\begin{equation}
	\label{eq:lem1}
	\vz^*_k = \vx^*_k + \veta_k, ~ k \in \CK.
	\end{equation}
	We plug this particular vector into
	\eqref{eqn:org_obj_dp}-\eqref{eqn:org_lastcon_dp} and observe that
	$\vv_k \geq \zv$, $k \in \CK$. Thus, $(\vx^*_k)_{k \in \CK}$ is a
	feasible solution for
	\eqref{eqn:org_model_obj}--\eqref{eqn:org_model_nonneg}. Next we
	plug
	\begin{equation}
	\label{eq:lem2}
	\begin{array}{ll}
	\vbeta^* = \valpha^*, & \\
	\vbeta^*_k = \valpha^*_k + \vxi_k, & k \in \CK,\\
	\vnu^*_k = \vlambda_k^*,  & k \in \CK,
	\end{array}
	\end{equation}
	into \eqref{eqn:dual_obj_dp}--\eqref{eqn:dual_nonneg_dp} and note
	that $\vtheta_k \geq \zv$, $k \in \CK$. This shows that
	$(\valpha^*, \valpha^*_k, \vlambda^*_k)_{k \in \CK}$ is a feasible
	solution for
	\eqref{eqn:dual_model_obj}--\eqref{eqn:dual_model_nonneg}. Consequently,
	we have feasible solutions for both the primal problem and the
	dual problem. When we consider the equalities in
	\eqref{eqn:org_obj_dp}-\eqref{eqn:org_lastcon_dp} and
	\eqref{eqn:dual_obj_dp}--\eqref{eqn:dual_nonneg_dp}, we obtain for
	$k \in \CK$ that
	\begin{equation}
	\label{eq:lem3}
	\begin{array}{l}
	\vv_k^* = \vc_k + \mB_k\veta_k - \mB_k\vz_k^* = \vc_k - \mB_k\vx^*_k, \\
	\vtheta^*_k = \mA_k\tr \vbeta^* + \mB_k\tr \vbeta_k^* + \vnu_k^* - \vr_k - \mB_k\tr\vxi_k = \mA_k\tr \valpha^* + \mB_k\tr \valpha_k^* + \vlambda_k^* - \vr_k.
	\end{array}
	\end{equation}
	Recall that the strong duality of linear programming implies
	\[
	\sum_{k \in \CK}(\vr_k + \mB_k\tr\vxi_k)\tr\vz^*_k + \sum_{k \in \CK} \vxi_k\tr\vv^*_k = (\vc + \dsum_{k \in \CK}\mA_k \veta_k)\tr \vbeta^* +  \dsum_{k \in \CK}
	(\vc_k + \mB_k\veta_k)\tr \vbeta^*_k + \sum_{k\in\CK} (\ov +
	\veta_k)\tr \vnu^*_k - \dsum_{k \in \CK} \veta_k\tr \vtheta^*_k.
	\]
	Rewriting this equality with \eqref{eq:lem1}, \eqref{eq:lem2} and \eqref{eq:lem3} shows that
	\[
	\dsum_{k \in \CK} \vr_k\tr \vx^*_k = \vc\tr \valpha^* +  \dsum_{k \in \CK} \vc_k\tr \valpha^*_k + \sum_{k\in\CK} \ov \tr \vlambda^*_k.
	\]
	This establishes that $(\vx_k)_{k \in \CK}$ and
	$(\valpha^*, \valpha^*_k, \vlambda^*_k)_{k \in \CK}$ are the
	primal and dual optimal solutions for
	\eqref{eqn:org_model_obj}--\eqref{eqn:org_model_nonneg},
	respectively. The desired equalities in the hypothesis follow from
	our construction.
\end{proof}

\thmnum{\ref{thm:equiv}}{
    Let $(\vu_k^*, \vw_k^*)_{k \in \CK}$ and
  $(\vgamma^*, \vgamma_k^*)_{k \in \CK}$ be the primal and dual
  optimal solutions of
  \eqref{eqn:pri_obj}-\eqref{eqn:pri_nonneg}. Using again the primal
  and dual optimal solutions, $(\vx_k^*)_{k \in \CK}$ and
  $(\valpha^*, \valpha^*_k)_{k \in \CK}$ of the original problem
  \eqref{eqn:org_model_obj}--\eqref{eqn:org_model_nonneg}, we obtain
  \[
    \begin{array}{cl}
    Z = \bar{Z} - \sum_{k \in \CK} \vr_k\tr \veta_k - \sum_{k \in \CK} (\vc_k + \mB_k\veta_k)\tr\vxi_k, &\\
      \vx_k^* = \mD_k\tr \vu_k^* - \veta_k, & k \in \CK, \\
      \valpha^* = \vgamma^*, & \\
      \valpha_k^* = \mF_k\tr\vgamma^*_k - \vxi_k, & k \in \CK.
    \end{array}
  \]
}

\begin{proof}
  To obtain the linear programming model
  \eqref{eqn:pri_obj}-\eqref{eqn:pri_nonneg}, we apply for
  $k \in \CK$ the change of variables $\mD_k\tr \vu_k = \vz_k$ and
  $\mE_k\tr \vw_k = \vv_k$ to the model
  \eqref{eqn:org_obj_dp}-\eqref{eqn:org_lastcon_dp}. Likewise,
  multiplying both sides of the equality constraints
  \eqref{eqn:org_ck_dp} with $\mF_k$ leads for $k\in\CK$, to the
  change of variables $\vbeta_k = \mF_k\tr\vgamma_k$. Note that both
  sides of the constraints
  \eqref{eqn:org_ub_dp}-\eqref{eqn:org_lastcon_dp} are multiplied by
  $M$-matrices, and hence, feasibility is not affected. Using next
  Lemma \ref{lem:output} implies
  \[
    \begin{array}{cl}
      \vx_k^* = \vz_k^* - \veta_k = \mD_k\tr \vu_k^* - \veta_k, & k \in \CK, \\
      \valpha^* = \vbeta^*= \vgamma^*, & \\
      \valpha_k^* = \vbeta_k^* - \vxi_k = \mF_k\tr\vgamma^*_k - \vxi_k, & k \in \CK.
    \end{array}
  \]
  \citet[Proposition 1]{mangasarian2011privacy} has shown that the
  optimal objective function values of
  \eqref{eqn:org_obj_dp}-\eqref{eqn:org_lastcon_dp} and
  \eqref{eqn:pri_obj}-\eqref{eqn:pri_nonneg} are the same. Recall from
  the proof of Lemma \ref{lem:output} that
  \eqref{eqn:org_obj_dp}-\eqref{eqn:org_lastcon_dp} is obtained from
  \eqref{eqn:org_model_obj}--\eqref{eqn:org_model_nonneg} by applying
  for $k \in \CK$, the transformations $\vz_k = \vx_k + \veta_k$ and
  $\vbeta_k = \valpha_k + \vxi_k$. Using the first transformation, the
  constant term $\sum_{k \in \CK} \vr_k\tr \veta_k$ is subtracted from
  the objective function. Moreover, the same transformation also alters
  the right-hand-side of \eqref{eqn:org_model_ck} as
  $\vc_k + \mB_k\veta_k$, $k \in \CK$. The second transformation with
  this new right-hand-side subtracts additionally the constant term
  $\sum_{k \in \CK} (\vc_k + \mB_k\veta_k)\tr\vxi_k$ from the
  objective function. Adding both constant terms establishes the
  required equality:
  \[
    \bar{Z} = Z + \sum_{k \in \CK} \vr_k\tr \veta_k + \sum_{k \in \CK}
    (\vc_k + \mB_k\veta_k)\tr\vxi_k.
  \]
  This completes the proof.
\end{proof}

\lemnum{\ref{lem:security}}{
  Suppose for $k\in \CK$ that $1 \leq m < n_k \leq s_k$,
  $1 < m_k \leq t_k$, and both $\mA_k$ and $\mB_k$ have full rank.
  Even if all private information of party $k\in \CK$ are known, then
  finding any one of $\mD_k$, $\mF_k$, $\veta_k$ or $\vxi_k$ requires
  obtaining a particular solution to a system of linear equations with
  infinitely many solutions.
}
\begin{proof}
  Using \eqref{eqn:notation}, we first check the relations that
  involve $\mD_k$ as the only unknown. This leaves us with
  $\bar{\mA}_k = \mA_k\mD_k\tr$, where $\mA_k$ is $m \times n_k$
  matrix. Since $m < n_k$ and $\text{rank}(A_k) = m$, this system has
  infinitely many solutions. In all other relations listed in
  \eqref{eqn:notation}, $\mD_k$ is placed along with another private
  random matrix. However, we note that when $n_k=s_k$, we have
  \[
    \bar{\ov}_k - \bar{\mG}_k\bar{\mH}_k\inv\bar{\veta}_k = \mG_k\ov +
    \mG_k\veta_k - \bar{\mG}_k\bar{\mH}_k\inv\bar{\veta}_k = \mG_k\ov.
  \]
  When $n_k=1$, the random matrix $\mG_k$, and consequently, $\mD_k$
  can be obtained. However, we have assumed that $n_k > 1$ leading to
  an underdetermined system. Thus, $\mD_k$ cannot be obtained without
  solving a system with infinitely many solutions. In a similar vein,
  $\veta_k$ appears only in $\tilde{\veta}_k = \mA_k\veta_k$ as the
  sole unknown, but again this system has infinitely many solution for
  $m < n_k$. Matrices $\mF_k$ and $\vxi_k$ do not directly appear in
  any one of the equations without being multiplied with another
  random matrix. Again we note that when $n_k=s_k$, we have
  \[
    \bar{\vc}_k - \bar{\mB}_k\bar{\mH}_k\inv\bar{\veta}_k = \mF_k\vc_k +
    \mF_k\mB_k\veta_k - \bar{\mB}_k\bar{\mH}_k\inv\bar{\veta}_k = \mF_k\vc_k.
  \]
  This system is also undetermined, since $m_k > 1$. Finally, when
  $m_k = t_k$, we have
  \[
    \bar{\vr}_k -
    \bar{\mB}\tr_k\bar{\mF}^{-\intercal}_k\bar{\vxi}\tr_k = \mD_k\vr_k
    + \mD_k\mB_k\tr\vxi_k -
    \bar{\mB}\tr_k\bar{\mF}^{-\intercal}_k\bar{\vxi}\tr_k = \mD_k\vr_k.
  \]
  Given $1 < n_k \leq s_k$, this last system has infinitely many
  solutions as well.
\end{proof}

\end{document}